\documentclass{article}
\usepackage{indentfirst}
\usepackage{amsthm}
\usepackage{amsfonts}
\usepackage{amssymb}
\usepackage{xcolor}

\newcommand{\COLORON}{1}
\newcommand{\NOTESON}{0}
\newcommand{\Debug}{0}

\newcommand{\pl}{piecewise linear}
\newcommand{\ple}{piecewise linear embedding}
\newcommand{\pd}{properly discontinuous}
\newcommand{\lfl}{locally flat}
\newcommand{\CtS}{Cantor 3-sphere}
\newcommand{\vKT}{van Kampen's theorem}

\newcommand{\GeThm}{Geometrization Theorem}

\newcommand{\new}{\mymargin{New}}

\usepackage{amsthm,amssymb,amsmath,bbm,enumerate,graphicx,epsf,stmaryrd}
\usepackage[bookmarks, colorlinks=false, breaklinks=true]{hyperref} 

\usepackage{authblk}

\newcommand{\comment}[1]{}
\newcommand{\COMMENT}[1]{}

\definecolor{darkgray}{rgb}{0.3,0.3,0.3}
\newcommand{\defi}[1]{{\color{darkgray}\emph{#1}}}

\newcommand{\acknowledgement}{\section*{Acknowledgement}}

\comment{
	\begin{lemma}\label{}	
\end{lemma}
\begin{proof}

\end{proof}

\begin{theorem}\label{}
\end{theorem} 
\begin{proof} 	

\end{proof}

}

\newtheorem{proposition}{Proposition}[section]
\newtheorem{definition}[proposition]{Definition}
\newtheorem{theorem}[proposition]{Theorem}
\newtheorem{corollary}[proposition]{Corollary}

\newtheorem{lemma}[proposition]{Lemma}

\newtheorem{examp}[proposition]{Example}

\newtheorem{remark}{Remark}
\newtheorem{question}[proposition]{Question}

\newcommand{\FIG}{0}

\ifnum \NOTESON = 1 \newcommand{\note}[1]{ 

\hspace*{-30pt}
	{\color{blue}  NOTE: \color{Turquoise}{\small  \tt \begin{minipage}[c]{1.1\textwidth}  #1 \end{minipage} \ignorespacesafterend }} 
	
	}
\else \newcommand{\note}[1]{} \fi

\newcommand{\afsubm}[1]{ \ifnum \Debug = 1 {\mymargin{#1}}
\fi} 

\ifnum \Debug = 1 
\else  \fi

\ifnum \FIG = 1 \newcommand{\fig}[1]{Figure ``{#1}''}
\else \newcommand{\fig}[1]{Figure~\ref{#1}} \fi

\ifnum \FIG = 1 
\else  \fi

\ifnum \Debug = 1 \usepackage[notref,notcite]{showkeys}
\fi

\ifnum \COLORON = 0 \renewcommand{\color}[1]{}
\fi

\newcommand{\R}{\ensuremath{\mathbb R}}

\newcommand{\Z}{\ensuremath{\mathbb Z}}

\newcommand{\BS}{\ensuremath{\mathbb S}}

\newcommand{\ce}{\ensuremath{\mathcal E}}

\newcommand{\cv}{\ensuremath{\mathcal V}}

\newcommand{\sm}{\backslash}

\newcommand{\isom}{\cong}

\newcommand{\cls}[1]{\ensuremath{\overline{#1}}}
\newcommand{\act}{\curvearrowright}

\makeatletter
\DeclareRobustCommand{\cev}[1]{%
  \mathpalette\do@cev{#1}%
}
\newcommand{\do@cev}[2]{%
  \fix@cev{#1}{+}%
  \reflectbox{$\m@th#1\vec{\reflectbox{$\fix@cev{#1}{-}\m@th#1#2\fix@cev{#1}{+}$}}$}%
  \fix@cev{#1}{-}%
}
\newcommand{\fix@cev}[2]{%
  \ifx#1\displaystyle
    \mkern#23mu
  \else
    \ifx#1\textstyle
      \mkern#23mu
    \else
      \ifx#1\scriptstyle
        \mkern#22mu
      \else
        \mkern#22mu
      \fi
    \fi
  \fi
}

\makeatother

\newcommand{\pth}[2]{\ensuremath{#1}\text{--}\ensuremath{#2}~path}

\newcommand{\g}{\ensuremath{G\ }}
\newcommand{\G}{\ensuremath{G}}

\newcommand{\Cg}{Cayley graph}
\newcommand{\Cc}{Cayley complex}

\newcommand{\Lr}[1]{Lemma~\ref{#1}}

\newcommand{\Tr}[1]{Theorem~\ref{#1}}

\newcommand{\Sr}[1]{Section~\ref{#1}}

\newcommand{\Cr}[1]{Corollary~\ref{#1}}

\newcommand{\Dr}[1]{De\-fi\-nition~\ref{#1}}

\renewcommand{\iff}{if and only if}
\newcommand{\fe}{for every}
\newcommand{\Fe}{For every}

\newcommand{\st}{such that}

\newcommand{\ti}{there is}

\newcommand{\obda}{without loss of generality}

\newcommand{\wrt}{with respect to}

\newcommand{\tfae}{the following are equivalent}

\newcommand{\labtequ}[2]{
 \begin{equation} \label{#1} 	\begin{minipage}[c]{0.9\textwidth}  #2 \end{minipage} \ignorespacesafterend \end{equation} } 

\newcommand{\labtequc}[2]{ \begin{equation} \label{#1} 	\text{   #2 } \end{equation} }

\newcommand{\mymargin}[1]{
 \ifnum \Debug = 1
  \marginpar{%
    \begin{minipage}{\marginparwidth}\small%
      \begin{flushleft}%
        {\color{blue}#1}%
      \end{flushleft}%
   \end{minipage}%
  }%
 \fi
}%

\newcommand{\mySection}[2]{}

\begin{document}
	\title{Discrete group actions on 3-manifolds and embeddable Cayley complexes}
	
\author[1]{Agelos Georgakopoulos\thanks{Supported by  ERC grant 639046 and EPSRC grants EP/V048821/1 and EP/V009044/1.}}
\affil[1,2]{  {Mathematics Institute}\\
 {University of Warwick}\\
  {CV4 7AL, UK}}
\author[2]{George Kontogeorgiou}

\date{\today}

\maketitle

\newcommand{\bsig}{\ensuremath{\boldsymbol{\sigma}}}

\newcommand{\sico}{simply connected}
\newcommand{\prs}{planar rotation system}
\newcommand{\gcc}{generalised \Cc}

\begin{abstract}
We prove that a group $\Gamma$ admits a discrete topological (equivalently, smooth) action on some simply-connected 3-manifold if and only if $\Gamma$  has a Cayley complex embeddable ---with certain natural restrictions--- in one of the following four 3-manifolds: (i) $\mathbb{S}^3$, (ii) $\mathbb{R}^3$, (iii)  $\mathbb{S}^2 \times \R$, (iv) the complement of a tame Cantor set in $\mathbb{S}^3$. 
\end{abstract}



\begin{center}{\begin{minipage}{.65 \linewidth}Dedicated to the memory of Mikis Theodorakis, whose oratorio {\it Axion Esti} rhymes with `{\it action $\BS^3$}'. 
\end{minipage}}
\end{center}

\section{Introduction}

The main result of this paper is that a group $\Gamma$ admits a discrete topological action on some simply-connected 3-manifold if and only if $\Gamma$  has a Cayley complex embeddable ---with certain natural restrictions--- in one of the following four 3-manifolds: (i) $\mathbb{S}^3$, (ii) $\mathbb{R}^3$, (iii)  $\mathbb{S}^2 \times \R$, (iv) the complement of a tame Cantor set in $\mathbb{S}^3$. Which of these four cases occurs is determined by the number of ends of $\Gamma$ \cite{Hopf}. By a \defi{discrete} action we mean a faithful, properly discontinuous, co-compact, topological action. All our manifolds and actions are topological, but they can be smoothed by a classical result of Bing and Moise \cite{bing1959,Moise52}, and a recent result of Pardon \cite{Pardon21, PardonMO} (see \Tr{pardon new}).

A homeomorphic image $C$ of the Cantor set in $\BS^3$ is called \defi{tame}, if it is contained in a piecewise linear arc. It is known that if $C'$ is another tame Cantor set in $\BS^3$, then $\BS^3 - C$ is homeomorphic to $\BS^3 - C'$; see \cite{SkoCan} and references therein. A topological space homeomorphic to $\BS^3 - C$ will be called a \defi{\CtS}. Its importance is established by the following result:

\begin{theorem} \label{Agol}
Let $M$ be a connected, simply connected, topological 3-manifold. Suppose $M$ admits a properly-discontinuous, co-compact action by homeomorphisms. Then $M$ is homeomorphic to one of the following four spaces:  
$\BS^3$, or $\R^3$, or $\BS^2 \times \R$, or the \CtS.
\end{theorem}

The special case of \Tr{Agol} when the action is smooth and free became well-known to experts in the aftermath of the Thurston--Perelman \GeThm\ \cite{ThuThr,PerEnt,PerRic}. A proof for the general case, which is due to others, can be found in the Appendix.

\subsection{Finite group actions on $\BS^2$ and $\BS^3$}

The following is a classical theorem, essentially going back to a 1896 paper of Maschke \cite{Maschke}; for modern references see e.g.\ \cite[Theorem 1.16.]{BaAut} and \cite{Kleinian}.

\begin{theorem}[Folklore] \label{Maschke} 
A finite group admits a faithful action (by homeomorphisms or isometries) on $\BS^2$ \iff\ it has a planar \Cg.
\end{theorem}

The finite case of our result is a 3-dimensional version of \Tr{Maschke}, replacing $\BS^2$ by $\BS^3$. We will prove that a finite group $\Gamma$ admits a faithful ---topological, smooth, or isometric--- action  on $\BS^3$, if and only if it has a Cayley complex $X$ that embeds topologically in $\BS^3$ so that the canonical action  $\Gamma \act X$ maps each chamber boundary to a chamber boundary. Before clarifying the details, we recall that this class of groups is now well-understood: 

\begin{theorem} \label{O4}
A finite group admits a faithful action by homeomorphisms/smooth maps/isometries on $\BS^3$ \iff\ it is isomorphic to a subgroup of the orthogonal group $O(4)$.
\end{theorem}

Indeed, using the \GeThm, Dinkelbach and Leeb \cite{DL09} showed that every finite smooth action on $\BS^3$ is conjugate to an isometric action. Pardon \cite{Pardon21} complemented this by proving that every topological action on $\mathbb{S}^3$ is the uniform limit of smooth actions. Since any isometry of $\BS^3$ extends to an element of $O(4)$, these two facts combined establish that the finite groups admitting faithful, topological actions on $\BS^3$ coincide, up to isomorphism, with the finite subgroups of $O(4)$, as anticipated by Zimmermann \cite{Zimmermann18}. This forms a quite rich family of groups, described explicitly in \cite{ConwaySmith}. \Tr{O4} is the culmination of a long effort,   the history of which is surveyed e.g.\ in \cite{DM85,edmonds83,Hambleton14,Zimmermann11}. 
In contrast, the groups of \Tr{Maschke} were described by Maschke in 1986: they are just the finite cyclic and dihedral groups, and 8 sporadic ones \cite[Figures 1--10]{Maschke}, \cite{GrossTucker}.

For an analogue of \Tr{Maschke} for actions on $\BS^3$, \Cg s are unlikely to suffice: every finite graph embeds in $\BS^3$. As we will see, the key is to consider \Cc es instead. Define a \defi{\gcc} of a group $\Gamma$ to be a simply connected 2-dimensional cell-complex $X$, \st\ \ti\ an action of $\Gamma$ on $X$ that is regular on the 0-skeleton $X^0$. This generalises the standard notion of  \Cc, in that we allow the action to fix 1-cells and 2-cells.

Any embedding $\phi$ of a 2-dimensional cell-complex $X$ into $\BS^3$ ---more generally, into an orientable 3-manifold--- induces a cyclic ordering $\sigma_e$ of the 2-cells containing a given 1-cell $e$ of $X$, e.g.\ by considering the clockwise cyclic order in which these 2-cells intersect a disc locally perpendicular to $e$. This family of cyclic orderings $\sigma(\phi):= \{\sigma_e\}_{e\in X^1}$ will be called a \prs, defined more carefully in \Sr{sec rot sys}. Importantly, $\sigma(\phi)$ contains all the information needed to determine which sets of 2-cells of $X$ bound a chamber of $\phi(X)$; these sets will be called \defi{pre-chambers}. Given a cellular group action $\Gamma \act X$, e.g.\ when $X$ is a \Cc\ of   $\Gamma$, we say that $\sigma(\phi)$ is \defi{$\Gamma$-invariant} if $\Gamma \act X$ preserves  $\sigma(\phi)$ (see \Sr{invariant prs} for details). 


\medskip
Our 3-dimensional analogue of \Tr{Maschke} is the equivalence of items \ref{m i} and \ref{m iii} of the following more general statement:

\begin{theorem} \label{main thm}
For a finite group $\Gamma$ \tfae: 
\begin{enumerate}
\item \label{m i} $\Gamma$ is one of the groups of \Tr{O4};
\item \label{m ii} $\Gamma$ admits a \gcc\ $X$ with a $\Gamma$-invariant \prs; 
\item \label{m iii} $\Gamma$ admits a \gcc\ $X$ with an embedding $\phi: X \to \BS^3$ with $\Gamma$-invariant \prs; 
\item \label{m iv} $\Gamma$ admits a \gcc\ $X$ with an  embedding $\phi: X \to \BS^3$ \st\ $\phi(X)$ is invariant 
under some faithful topological action of $\Gamma$ on $\BS^3$ which acts regularly on the vertices of $X$.\footnote{In other words, \ref{m iv} states that $\phi$ is equivariant \wrt\  actions of $\Gamma$ on $X$ and $\BS^3$ with the desired   properties.} 
\end{enumerate}
\end{theorem}

As an example, let $\Gamma$ be the cartesian product of two finite cyclic groups $C_k, C_\ell$, and consider its  generalised \Cc\ $X$ \wrt\ the standard presentation $\left< a,b \mid a^k, b^\ell, [ab] \right>$. Thus $X$ is a quadrangulated torus $T$, united with 2-cells bounding its essential cycles spanned by each one of the generators $a,g$. It is easy to see a  topological embedding of $X$ in $\BS^3$, with all $a$-coloured 2-cells inside $T$ and all $b$-coloured 2-cells outside it. Notice that the \prs\ induced by this embedding is $\Gamma$-invariant, as required by item \ref{m iii}. The reader will be able to see a topological action of $\Gamma$ that preserves this embedding, as postulated by item \ref{m iv}, which trivially implies \ref{m i}. The implication  \ref{m ii} $\to$ \ref{m iii} says that we could have specified the \prs\ abstractly, as a combinatorial set of cyclic orderings, without mention to a particular embedding. This implication makes use of the validity of the Poincar\'e conjecture via a result of Carmesin \cite{CarEmbII}  (\Tr{JCtheorem} below). In \Sr{sec ii to iv} we observe that Carmesin's result fails in  general for infinite 2-complexes, but remains true for \Cc es, a fact that relies on the \GeThm\ via \Tr{Agol}.



\subsection{Infinite discrete actions on open 3-manifolds}

\Tr{main thm} is the specialisation of our main theorem to finite groups. We call the four 3-manifolds featuring in \Tr{Agol} the \defi{special} 3-manifolds. They also feature in our main theorem:

\begin{theorem} \label{inf thm}
For a finitely generated group $\Gamma$ \tfae: 
\begin{enumerate}
\item \label{inf i} $\Gamma$ admits a faithful, properly discontinuous, co-compact, topological action on a simply connected 3-manifold; 
\item \label{inf ii} $\Gamma$ admits a \gcc\ $X$ with a $\Gamma$-invariant \prs\ with finite pre-chambers; 
\item \label{inf iii} $\Gamma$ admits a \gcc\ $X$ with an 
embedding $\phi: X \to \BS^3$ \st\ $\sigma(\phi)$ is $\Gamma$-invariant and has finite pre-chambers; 
\item \label{inf iv} $\Gamma$ admits a \gcc\ $X$ with an  embedding $\phi$ into a special 3-manifold $M$ 
 \st\  $\phi(X)$ is invariant under some faithful, properly discontinuous, co-compact, topological action of $\Gamma$ on $M$, which acts regularly on the vertices of $X$. Moreover,  $\phi(X)$ has finite pre-chambers. 

\end{enumerate}
\end{theorem}

An analogous 2-dimensional statement, generalising 
\Tr{Maschke}, can be found in \cite[Theorem~1.1]{Kleinian}. It says that the infinite groups acting discretely on a planar surface are exactly the Kleinian function groups. They also coincide with those groups admitting a Cayley graph with invariant \defi{\prs}. Details can be found in  \cite{BowMas,Kleinian,LevMasSpe}. 

We saw examples of discrete actions on $\BS^3$ above. Groups acting discretely on $\R^3$ include lattices in the Thurston geometries homeomorphic to $\R^3$, such as euclidean and hyperbolic crystalographic groups, and the discrete Heisenberg group. On the contrary, $\Z^4$ cannot act discretely on $\R^3$, by a theorem of Stallings saying that it is not isomorphic to any subgroup of a 3-manifold group \cite{BoiARO}. Groups acting discretely on $\mathbb{S}^2 \times \R$ include $\Z$, and its cartesian product with any of the groups of \Tr{Maschke}. An example of a group acting discretely on the \CtS\ is the free group $F_r$ of rank $r \geq 2$. The universal cover of any closed 3-manifold with fundamental group $F_r$ provides an example, and such manifolds can be easily obtained using  connect sums of copies of $\BS^2 \times \BS^1$. 


\medskip
We now discuss the tightness of the various conditions of  \Tr{inf thm}. Item \ref{m iv} provides the most detailed information about these groups, and it trivially implies the other three (\Sr{conc}). The conditions in the other three items are necessary in the sense that dropping any one of them would violate at least one of the equivalences. To see that the condition of finite pre-chambers is necessary for the implications \ref{m ii},\ref{m iii} $\to$ \ref{m iv},\ref{m i}, consider the group $\Z^2$. Its standard \Cc\ embeds in $\R^3$ with invariant  \prs, but with two infinite pre-chambers, and it is well-known that $\Z^2$ is not one of the groups of \Tr{inf thm} \cite{MOfreeabelian}. We do not have an explicit example of a group with a \Cc\ that embeds in $\BS^3$ with finite pre-chambers but only with non-invariant \prs, but expect that such a group can be found using ideas of \cite[\S 10]{Kleinian}.

Dropping some of our conditions in pairs can lead to interesting statements, some of which are implicit in our proofs. A notable example is that the implication \ref{m i} $\to$ \ref{m iii} generalises to group actions on non-\sico\ 3-manifolds, resulting to embedded 2-complexes with an appropriate action by the same group; see \Tr{cor i to ii}. Dropping the condition of finiteness of pre-chambers leads to an interesting class of groups; see \Sr{conc} for related open problems.

In items \ref{inf i} and  \ref{inf iv} of \Tr{inf thm}, we can assume that the action is in addition smooth. \mymargin{If $G$ is hyperbolic and 1-ended, we can assume that the action is by isometries on $\mathbb{H}^3$.}


\medskip
This paper is structured as follows. \Sr{prels} provides all the necessary definitions. The implication \ref{m ii} $\to$ \ref{m iv} of \Tr{inf thm} (and \Tr{main thm}) is proved in \Sr{sec ii to iv}. 
Since the implication \ref{m iii} $\to$ \ref{m ii} is trivial, this also establishes the implication \ref{m iii} $\to$ \ref{m iv}. 
The implication \ref{m i} $\to$ \ref{m iv} is proved in \Sr{sec i to iv}, while  \ref{m iv} trivially implies the other three statements. \Sr{conc} puts the pieces together to conclude the proof of \Tr{inf thm}. 

In \Sr{sec iii to iv} we provide a separate proof of the  implication \ref{m iii} $\to$ \ref{m iv} of \Tr{main thm} avoiding the Poincar\'e conjecture. Thus we are able to prove the equivalence of items \ref{m i}, \ref{m iii} and \ref{m iv} of \Tr{main thm} without relying on Perelman's work. This is not the case for \Tr{inf thm}, where we do not see a way to avoid using the \GeThm.

\comment{\small

In the 50's and 60's, the classic papers of Bing(\cite{bing1952}, \cite{bing1964}),  Montgomery-Zippin(\cite{MZ54}), and Alford(\cite{Alford66}) hosted the first examples of non-smoothable finite topological actions over $\mathbb{S}^3$. Despite this hurdle, research on the properties of finite groups that admit a faithful topological action over $\mathbb{S}^3$ continued in a most fruitful manner. Moise(\cite{Moise52}) and Bing(\cite{bing1959}) proved that every homeomorphism between smooth manifolds is the uniform limit of a sequence of diffeomorphisms. Later, it was proven  that \textbf{free} finite topological actions on spheres are smoothable, see Madsen-Thomas-Wall(\cite{MTW76}). 

}

\section{Definitions and preliminary results} \label{prels}

\subsection{Group actions} \label{sec GA}


A 3-manifold is a topological space each point of which has a neighbourhood homeomorphic to an open subset of $\R^3$.

A (topological) \defi{action} $\Gamma\act M$ of a group $\Gamma$ on a topological space $M$ is a homomorphism from $\Gamma$ into the group of homeomorphisms of $M$.
Given such an action, the images of a  point $x\in M$ under $\Gamma\act M$ form the \defi{orbit} of $x$. 

An action $\Gamma \act M$ is \defi{faithful}, if for every two distinct $g,h \in \Gamma $ there exists an $x \in M$ such that $gx \neq hx$; or equivalently, if for each $g \neq e \in \Gamma $ there exists an $x \in M$ such that $gx \neq x$. It is \defi{free} if $gx \neq hx$ \fe\ $g,h\in \Gamma$ and $x\in M$. It is \defi{transitive} if \fe\ $x,y\in M$ \ti\ $g\in \Gamma$ with $gx =y$ (we will only encounter transitive actions on discrete spaces $M$). Finally, $\Gamma \act M$ is \defi{regular} if it is free and transitive.

An action $\Gamma \act M$ is \defi{\pd}, if  \fe\ compact subspace $K$ of $M$, the set $\{g\in\Gamma \mid g K \cap K \neq \emptyset\}$ is finite. It is \defi{co-compact}, if the quotient space $M/\Gamma$ is compact. If $M$ is locally compact, then an equivalent condition is that there is a compact subset $K$ of $M$ such that $\bigcup \Gamma K =M$. 


\subsection{Graphs} \label{graphs}

A (simple) \defi{graph} \g is a pair $(V,E)$ of sets, where $V$ is called the set of \defi{vertices}, and $E$ is a set of two-element subsets of $V$, called the set of \defi{edges}. We will write $uv$ instead of $\{u,v\}$ to denote an edge. A \defi{multi-graph} is defined similarly, except that $E$ is a multi-set, and it can have elements consisting of just one vertex. 

We let $V(G)$ denote the set of vertices of a graph $G$, and $E(G)$ denote the set of edges of $G$. 
 

Every (multi-)graph $\g=(V,E)$ gives rise to a $1$-complex, by letting $V$ be the set of 0-cells, and for each $uv\in E$ introducing an arc with its endpoints identified with $u$ and $v$. Thus we will sometimes interpret the word \defi{graph} as a 1-complex. In particular, when discussing \defi{embeddings} of graphs we will mean topological embeddings of 1-complexes.


A \defi{generalised \Cg} of a group $\Gamma$ is a graph $G$ endowed with an action $\Gamma \act G$ by isomorphisms which action is regular on $V(G)$. This is analogous to our definition of a generalised \Cc\ in the introduction. Again the difference to the standard notion of a \Cg\ is that we allow the action to fix edges.

\subsection{2-complexes} \label{ccs}

A \emph{$2$-complex} is a topological space $X$ obtained as follows. We start with a $1$-complex $X^1$ as defined above, called the \defi{1-skeleton} of $X$. We then introduce a set $X^2$ of copies of the closed unit disc $\mathbb{D} \subseteq \R^2$,  called the \defi{2-cells} or \defi{faces} of $X$, and for each $f\in X^2$ we {attach} $f$ to $X^1$ via a map $\phi_f: \BS^1 \to X^1$, called the \defi{attachment map}, where we think of $\BS^1$ as the boundary of $\mathbb{D}$. Attaching here means that we consider the quotient space where each point $x$ of $\BS^1\subset f$ is identified with $\phi_f(x)$. We let $X^0:= V(X^1)$ be the set of \defi{vertices}, or 0-cells, of $X$.

We say that $X$ is \defi{regular}, if $\phi_f$ is a homeomorphism onto its image \fe\ $f\in X^2$. We say that  $X$ is \defi{edge-regular}, if each $\phi_f$ is injective on 1-cells, that is, $x\in X^0$ holds for every point $x\in X$ with more than one pre-image under $\phi_f$. 

For $f\in X^2$, we write $f =[x_1,\dots, x_k]$ if $x_1,\dots, x_k$ is the cyclic sequence of vertices appearing in the image of $\phi_f$. 

To each 2-cell $f=[x_1,\dots, x_k] \in X^2$, we associate two distinct \defi{directed 2-cells} 
$f_1,f_2$, also denoted by $\langle x_1,\dots, x_k\rangle$ and $\langle x_k,\dots, x_1\rangle$ respectively. Their \defi{reverses} are defined as $f_1^{-1}:= f_2$ and $f_2^{-1}:= f_1$.

\smallskip
If $X$ is not regular then it is always possible to produce a regular complex $X'$ homeomorphic to $X$ using the \emph{barycentric subdivision} defined as follows. 
For each edge $e=uv\in X^1$, we subdivide $e$ by adding a new vertex $m$ at its midpoint. For each occurrence of $e$ in a 2-cell $f$ of $X$, we replace that occurence by the pair $um, mv$ or $vm, mu$ as appropriate. We then triangulate each 2-cell $h=x_1,\ldots,x_k$ of the resulting $2$-complex by adding a new vertex $c$ in its interior, adding the edges $c x_1,\ldots, cx_k$ to the 1-skeleton, and replacing $h$ by the 2-cells $[c, x_1, x_2], [c, x_2, x_3], \ldots, [c,x_k, x_1]$. 

Note that the barycentric subdivision $X'$ of $X$ is a simplicial complex, in particular a regular one, and its 1-skeleton is a simple graph. Here, a 2-complex $X$ is \defi{simplicial}, if its 1-skeleton is a simple graph and each $F\in X^2$ is of the form $f=[x_1, x_2, x_3]$ where $x_1,x_2,x_3$ are distinct vertices. 

\smallskip
For each $v\in X^0$, the \emph{link graph} $L_X(v)$ is the graph on the neighbourhood of $v$ in $X^1$ with $uw \in E(L_X(v))$ if and only  if $u,v,w$ are consecutive in a 2-cell of $X$. Alternatively, we can define $L_X(v)$ so that its vertices are the edges incident with $v$, and two edges $vu,vw$ are joined by an edge of $L_X(v)$ whenever $u,v,w$ are consecutive in a 2-cell of $X$. These two definitions yield isomorphic graphs when $X^1$ is a simple graph, and it is a matter of convenience to use the one or the other. 
In general, link graphs are more naturally defined as multigraphs, as $u,v,w$  may be consecutive in more than one 2-cells of $X$.  

\subsection{Embeddings and Chambers} \label{chambers}

An \textit{embedding} of a space $X$ in a space $Y$ is homeomorphism between $X$ and a subspace of $Y$.

For an embedding $\phi:X\rightarrow M$ of a $2$-complex $X$ into a 3-manifold $M$, we call each connected component of $M \setminus \phi(X)$  a  \emph{chamber}. The \defi{boundary} $\partial C$ of a chamber $C$ is the set of points $x \in \phi(X)$ in the closure of $C$ that are not in the interior of $C$. Indeed, as $C$ is an open set, $\partial C$ is disjoint from $C$. The following basic fact helps to further explain the notion.

\begin{proposition}[{\cite[Proposition 2.1]{Whitney3D}}] \label{chambers}
Let $\phi:X\rightarrow \BS^3$ be an embedding of a finite,  $2$-complex $X$, \st\ every 0-cell and 1-cell of $X$ is contained in a 2-cell.  Then $\partial C$ is a union of 2-cells of $X$ \fe\ $\phi$-chamber $C$.\qed
\end{proposition}



\comment{
Suppose a group $\Gamma$ acts both on a manifold $M$ and a 2-complex $X$.  An embedding $f: X \to M$ is called \defi{$\Gamma$-equivariant} if the following diagram \mymargin{perhaps consider embedded complexes throughout and remove} commutes 
\catcode`\@=11
\newdimen\cdsep
\cdsep=3em

\def\cdstrut{\vrule height .6\cdsep width 0pt depth .4\cdsep}
\def\@cdstrut{{\advance\cdsep by 2em\cdstrut}}

\def\arrow#1#2{
	\ifx d#1
	\llap{$\scriptstyle#2$}\left\downarrow\cdstrut\right.\@cdstrut\fi
	\ifx u#1
	\llap{$\scriptstyle#2$}\left\uparrow\cdstrut\right.\@cdstrut\fi
	\ifx r#1
	\mathop{\hbox to \cdsep{\rightarrowfill}}\limits^{#2}\fi
	\ifx l#1
	\mathop{\hbox to \cdsep{\leftarrowfill}}\limits^{#2}\fi
}
\catcode`\@=12

\cdsep=3em
\[
\begin{matrix}
	\Gamma \times H                   & \arrow{r}{(id,f)}   & \Gamma \times M                    \cr
	\arrow{d}{\gamma} &                      & \arrow{d}{\phi} \cr
	H                 & \arrow{r}{f} & M                  \cr
\end{matrix}
\]  

where $\phi,\gamma$ denote the aforementioned actions.
}

\subsection{Local flatness} \label{loc flat}

We recall the standard notion of local flatness. 
An embedding $\phi: \BS^2 \to M$, where $M$ is a 3-manifold, is \defi{locally flat}, if  for each $x\in \phi(\BS^2)$ there exists a neighbourhood $U_x$ of $x$ \st\ the topological pair $(U_x,U_x\cap \phi(X))$ is homeomorphic to $(\mathbb{R}^3, \mathbb{R}^2)$, by which we mean that there is a homeomorphism from $U_x$ to $\mathbb{R}^3$ mapping $U_x\cap \phi(X)$ to $\mathbb{R}^2 \subset \mathbb{R}^3$. (A \defi{topological pair} $(X,A)$ consists of a topological space $X$ and a subspace $A\subseteq X$.)

We can extend the notion of local flatness to an embedding  $\phi: X \to M$ of a $2$-complex $X$ instead of $\BS^2$: we say that $\phi$ is \defi{locally flat}, if the restriction of $\phi$ to each homeomorphic image of $\BS^2$ in $X$ is locally flat, as defined above.

\subsection{Rotation systems} \label{sec rot sys}

A \defi{rotation system} of a graph $G$ is a family $(\sigma_v)_{v\in V(G)}$ of cyclic orderings of the edges incident with each vertex $v\in V(G)$. Every embedding of $G$ on an orientable surface defines a rotation system, by taking $\sigma_v$ to be the clockwise cyclic ordering in which the edges incident to $v$ appear in the embedding. The rotation system $(\sigma_v)_{v\in V(G)}$ is said to be \defi{planar}, if it can be defined by an embedding of $G$ in the sphere $\BS^2$. 

Let $X$ be an edge-regular $2$-complex, and let $\overleftrightarrow{E}(X)$ denote the set of \defi{directions} of its 1-cells, that is, the set of directed pairs $\overrightarrow{xy}:= \left<x,y \right>$ such that $xy\in X^1$. Thus every 1-cell gives rise to two elements of $\overleftrightarrow{E}(X)$. 
A \defi{rotation system} of  $X$ is a family $(\sigma_e)_{e \in \overleftrightarrow{E}(X)}$ of cyclic orderings $\sigma_e$ of the 2-cells incident with each  $e = \overrightarrow{xy} \in \overleftrightarrow{E}(X)$, such that if $e'=\overrightarrow{yx}$, then $\sigma_{e'}$ is the reverse of $\sigma_e$. A rotation system $(\sigma_e)_{e \in \overleftrightarrow{E}(X)}$ of  $X$ induces a rotation system $\sigma^v$ at each of its link graphs $L_X(v)$ by restricting to the directions of 1-cells emanating from $v$:  \fe\ $u\in V(L_X(v))$ we let $\sigma^v_{u}$ be the cyclic order obtained from 
$\sigma_{\overrightarrow{vu}}$ by replacing each 2-cell $f$ appearing in the latter by the edge $uw$ where $w$ is the unique neighbour $u$ in $V(L_X(v))$ such that $w,v,u$ appear consecutively in $f$.

A rotation system of a regular $2$-complex $X$ is \defi{planar}, if it induces a planar  rotation system on each of its link graphs.
Note that once, we fix an orientation, every locally flat embedding $\phi$ of $X$ into $\BS^3$ or $\R^3$ defines a planar rotation system, by letting $\sigma_e$ be the cyclic order in which the images of the 2-cells incident with $e$ appear in $U_x$, where $x$ is any interior point of $\phi(e)$, and $U_x$ is as in the definition of local flatness (\Sr{loc flat}).

\subsection{Invariant rotation systems} \label{invariant prs}

Suppose that a group $\Gamma$ acts on a $2$-complex $X$ by a faithful action $\Gamma \act X$. Let $\Sigma$ be the set of all rotation systems $\bsig= (\sigma_e)_{e \in \overleftrightarrow{E}(X)}$ on $X$ as defined in the previous subsection. Then we can let $\Gamma$ act on $\Sigma$ by elementwise multiplication as follows.  Recall that $\sigma_e$ is formally a ternary `betweenness' relation on the set $F(e)$ of 2-cells  containing $e$ \fe\ 1-cell $e$. For $g\in \Gamma$, we define the product $g \cdot \sigma_e:= \{ [ga,gb,gc]\mid [a,b,c]\in \sigma_e\}$, which is a cyclic ordering on $F(ge)$. This defines our action $g \cdot \bsig:= (g \cdot \sigma_e)_{e \in E(X)}$ of  $\Gamma$ on $\Sigma$. We will say that \bsig\ is $\Gamma$-\defi{invariant}, if $g \cdot \bsig$ coincides with $ \bsig$ up to a global change of orientation. To make this more precise, let $\eta: \Gamma \to \Z_2$ be a homomorphism from $\Gamma$ to the group $\Z_2$; we will use $\eta$ to carry the information of which $g\in \Gamma$ preserve/reverse the orientation. We say that \bsig\ is \defi{invariant} \wrt\ $\eta$, if
\labtequc{def cov}{$g \cdot \sigma_e =(-1)^{\eta(g)} \sigma_{ge}$,} 
holds \fe\ $e \in \overleftrightarrow{E}(X)$ and $g\in \Gamma$. 
We say that \bsig\ is \defi{$\Gamma$-invariant} if it is invariant with respect to some homomorphism $\eta$. Note that $\eta$ is uniquely determined by \bsig\ if it exists.  


\section{From Invariant Planar Rotation Systems to Invariant Cayley Complex Embeddings} \label{sec ii to iv}

The implication \ref{m ii} $\to$ \ref{m iii} of \Tr{main thm} is an immediate sequence of the following result of Carmesin:

\begin{theorem}[\cite{CarEmbII}] \label{JCtheorem}
A finite, simply connected, simplicial 2-complex admits an embedding $\phi$ in $\BS^3$ if and only if it admits a planar rotation system $\sigma$. Moreover, $\phi$ can be chosen so that $\sigma(\phi)$ coincides with $\sigma$.\footnote{The second statement is not stated explicitly in \cite{CarEmbII} but it is easily implied by the construction of the embedding.}
\end{theorem}

Indeed, if $X$ is a \gcc\ of $\Gamma$, and $\sigma$ a $\Gamma$-invariant \prs\ on $X$, then we can apply the barycentric subdivision to turn $X$ into a simplicial 2-complex $X'$, extend $\sigma$ to $X'$ in the obvious way, and apply \Tr{JCtheorem} to $X'$ to obtain an embedding of $X'$, which induces an embedding $\phi: X \to \BS^3$ with $\sigma(\phi)= \sigma$.

\begin{remark} \label{rem lfl}
If a locally-finite 2-complex $X$ admits an embedding $\phi$ in $\BS^3$ (or any 3-dimensional submanifold of $\BS^3$), then we may assume $\phi$ to be \lfl\ (as defined in \Sr{loc flat}). Indeed, we can modify $\phi$ to make it piecewise-linear (PL) \cite[Appendix C]{MaTaWaHar}, and it is easy to see that any PL embedding of $X$ in $\BS^3$ is \lfl. 
\end{remark}

The implication \ref{m ii} $\to$ \ref{m iii} of \Tr{inf thm} is more difficult, because \Tr{JCtheorem} does not extend to infinite 2-complexes:
\begin{theorem} \label{no inf Carm}
There is a locally-finite, simply connected, simplicial 2-complex  which admits a planar rotation system but does not admit an embedding in $\BS^3$. 
\end{theorem}
\begin{proof}
It is known that there is a contractible, hence simply connected, open 3-manifold $W$  which does not embed in any compact 3-manifold, let alone in $\BS^3$; see e.g.\ \cite{GuCon} and references therein. Let $T$ be a triangulation of such a manifold $W$. It is easy to see that if a simplicial complex $X$ has a topological embedding into some oriented 3-dimensional manifold, then it has a planar rotation system \cite{CarEmbII}. Letting $X$ be the 2-skeleton of $T$, we thus deduce that $X$ has a planar rotation system, since it embeds in $W$. 

We may assume \obda\ that there is no 3-cell $C$ of $T$ the boundary $\partial C$ of which separates $W$, because even if $T$ does not have this property its barycentric subdivision $T'$ will, and we could have chosen $T'$ instead of $T$. Thus no $\partial C$ separates $X$.

Suppose now that $X$ admits an embedding $f$ in $\BS^3$. By Remark~\ref{rem lfl}, we may assume that $f$ is \lfl. Then we can extend $f$ into an embedding of $W$ in $\BS^3$ as follows. For every  3-cell $C$ of $T$, we observe that $\partial C \subset X$ is homeomorphic to $\BS^2$, hence it separates $\BS^3$ into two components. One of these components $A$ is disjoint from $f(X)$, because $\partial C$ does not separate $X$ as mentioned above. Since $f$ is \lfl, $A$ is homeomorphic to $\R^3$ by the generalised Schoenflies theorem \cite{BrownGS, MazurGS}. Thus we can embed $C$ onto $A$. Doing so for each 3-cell $C$ of $T$, we obtain an embedding of $W$ into $\BS^3$, contradicting the choice of $W$.
\end{proof}

Despite the fact that \Tr{JCtheorem} fails for infinite 2-complexes in general, it does hold for Cayley-complexes:

\begin{theorem} \label{inf Carm}
Let $C$ be a finitely-presented \Cc\ admitting an invariant  planar rotation system $\sigma$ with finite pre-chambers. Then $C$ admits an embedding
$\phi$ into $\BS^3$  such that $\sigma(\phi)$ coincides with $\sigma$.
\mymargin{\tiny Moreover, if $C$ is 1-ended, then $\phi$ can be chosen to be accumulation-free in $\R^3$}
\end{theorem}
\begin{proof}
We follow the lines of the proof of \Tr{JCtheorem} in \cite{CarEmbII}, the main difference being that we apply  \Tr{Agol} instead of the Poincar\'e conjecture. This starts by defining a 3-manifold $M=T(C,\sigma)$, with 2-skeleton $C$ as follows. The \prs\ $\sigma$ induces a relation on the (directed) 2-cells of $C$, where two  directed 2-cells are \defi{related} if they appear consecutively is $\sigma_e$ for some common 1-cell $e$. The equivalence classes induced by this relation are called \defi{pre-chambers}. The intuition is that pre-chambers will coincide with the boundaries of chambers of the embedding of $C$ that we are constructing. For each pre-chamber $S$ thus defined, we attach a solid surface $\hat{S}$ to $C$, so that the attachment map ---which is not always injective--- maps the boundary of $\hat{S}$ onto $S$. This completes the definition of $M=T(C,\sigma)$, which Carmesin proves to be an oriented topological 3-manifold \cite[Lemma~4.5]{CarEmbII} (Carmesin works with finite $C$, but this proof extends verbatim to the infinite case). He then observes that $M$ is \sico\ if $C$ is (\cite[Lemma 4.6]{CarEmbII}); indeed, any loop in $M$ can be homotoped to one in $C$ by the construction of $M$. Finally, Carmesin observes that when $C$ is finite then $M$ is compact, hence homeomorphic to $\BS^3$ by the Poincar\'e conjecture! By construction, $C$ is embedded in $M\isom \BS^3$, and the rotation system of this embedding coincides with $\sigma$.

It remains to consider the case where $C$ is potentially infinite. 
Notice that the group of $C$ acts \pd ly and co-compactly on $C$, hence on $M$. Here we used the fact that $\sigma$ is invariant in order to extend the group action from $C$ to $M$.  Again, $M$ is \sico\ because $C$ is. Thus $M$ is homeomorphic to one of the 3-manifolds of \Tr{Agol}, each of which embeds in $\BS^3$.  
Again, since $C$ embeds in $M$ with rotation system $\sigma$ by construction, the statement follows.

\mymargin{\tiny If $C$ is 1-ended, then it homeomorphic to $\R^3$. Since $C$ is embedded in $M$ accumulation-free by definition, the second statement follows.}
\end{proof}

This establishes the implication \ref{inf ii} $\to$ \ref{inf iii} of \Tr{inf thm}. However, a closer inspection of the last proof reveals that we can obtain the stronger implication \ref{inf ii} $\to$ \ref{inf iv} (of both \Tr{inf thm} and \Tr{main thm}):

\begin{corollary} \label{cor inf Carm}
Let $C$ be a finitely-presented \Cc\ of a group $\Gamma$, admitting a $\Gamma$-invariant planar rotation system $\sigma$ with finite pre-chambers. Then $C$ admits an embedding $\phi: C \to M$ into a special 3-manifold $M$ \st\ $\phi(C)$ is invariant under some \pd, co-compact action $\Gamma \act M$.
\end{corollary}
\begin{proof}
Define $M=T(C,\sigma)$ as in the proof of \Tr{inf Carm}. Notice that pre-chambers are defined using $\sigma$, and so the action of $\Gamma$ on $C$ preserves pre-chambers. It follows that the action of $\Gamma$ on $C$ extends to an action on $M$ by homeomorphisms. As already observed,  $M$ is a special 3-manifold, and this action is \pd\ and co-compact.
\end{proof}

\mymargin{\small 

Remark: perhaps we can drop the finite pre-chambers condition if we work with manifolds $M$ with boundary.
}

\section{From Invariant Cayley Complex Embeddings to Group Actions} \label{sec iii to iv}

Since the implication \ref{m iii} $\to$ \ref{m ii} of \Tr{main thm} is trivial, the previous section also establishes the implication 
\ref{m iii} $\to$ \ref{m iv}. 
The aim of this section is to re-prove this implication \ref{m iii} $\to$ \ref{m iv} by a more elementary method that avoids the Poincar\'e conjecture. (We do not have a proof of the analogous implication of \Tr{inf thm} avoiding the Geometrization Theorem.) Thus the reader can choose to skip this section. 

\medskip
The purpose of this section is a proof of the following theorem without using the Poincar\'e conjecture.
\begin{theorem} \label{complex to action}
Let $\Gamma$ be a  finite group, let $X$ be a generalised Cayley complex of $\Gamma$, and  $\phi:X \to \BS^3$ an  embedding with  $\Gamma$-invariant rotation system $\sigma(\phi)$. Then there is a faithful topological action  $\Gamma \act \BS^3$ fixing $\phi(X)$ as a set, and acting regularly on its vertices. 
\end{theorem}

The idea is to reduce this to the following result of \cite{Whitney3D}. We say that a 2-complex $X$ is \defi{locally $k$-connected}, if each of its link graphs is $k$-connected. Recall that a graph is \defi{$k$-connected}, if it has more than $k$ vertices, and remains connected after removing any set of at most $k-1$ vertices.

\begin{theorem}[{\cite[Theorem~1.3]{Whitney3D}}] \label{thm Whitney3D}
Let $Y$ be a finite, simply-connected, locally 3-connected $2$-complex. Then, for every  two locally flat embeddings $\chi,\psi: Y \to \BS^3$, there exists a homeomorphism $\alpha : \BS^3 \to \BS^3$ such that $\psi= \alpha \circ \chi$.

Moreover, we may assume that $\alpha$ is determined by its restriction to $\chi(Y)$.\footnote{The second sentence is not explicitly stated in \cite{Whitney3D}, but it is an immediate consequence of the construction of $\alpha$ given there. Indeed, $\alpha$ is defined by extending  $\psi \circ \chi^{-1}$ from $\chi(Y)$ to all of $\BS^3$ as follows. It is proved that every chamber of $\chi(Y)$ and $\psi(Y)$ is bounded by a homeomorph of $\BS^2$, and the generalised Schoenflies theorem is then applied to map each chamber of $\chi$ to one of $\psi$. Thus the second sentence follows by always choosing the same outcome of the generalised Schoenflies theorem for a given homeomorphism between two copies of $\BS^2$ in $\BS^3$.}
\end{theorem}
Every finite Cayley complex automatically satisfies the simple connectedness condition, and it is locally 1-connected (\Lr{lem 1-con}), but it is not necessarily locally $3$-connected. With the lemmas that follow we will be able to increase the local connectedness of a complex $X$ as in \Tr{complex to action} by extending it to a super-complex. This super-complex $X'$ will inherit the canonical action of $\Gamma$, and its rotation system will still be $\Gamma$-invariant. This will allow us to apply \Tr{thm Whitney3D} to $X'$ to prove \Tr{complex to action}.

\bigskip
We start with the following basic fact about finite Cayley complexes. 

\begin{lemma}\label{lem 1-con}
Every finite generalised Cayley complex $X$ is locally 1-connected,  unless $X$ has fewer than 3 vertices.
\end{lemma}
This is well-known (see e.g.\ \cite[Lemma 5.1]{CarEmbII}), but we provide a proof for completeness:
\begin{proof}
Recall that the 1-skeleton $X^1$ of $X$ is a generalised Cayley graph, and so $X^1$ is 2-connected as it is finite. Thus any two incident edges of $X^1$ are contained in a cycle.

To prove that the link graph $L=L(o)$ of the identity element $o$ of $X$ is connected, pick two edges $e,f$ of $X$ incident with $o$. By the above remark, there is a cycle $C$ in $X^1$ containing both $e,f$. Consider a van Kampen diagram $K$ proving that $C$ is null-homotopic in $X$ using the 2-cells of $X$. Let $K_o$ denote the set of 2-cells appearing in $K$ that contain $o$. These 2-cells yield an \pth{e}{f}\ in $L$. Since $e,f$ where arbitrary edges of $o$, this proves that $L$, and hence every link graph of $X$, is connected. 
\end{proof}

Next, we show how to increase the local connectivity of an embedded 2-complex from 1 to 2 by passing to a super-complex:

\begin{lemma}\label{lem 2-con}
Let $X$ be a locally 1-connected $2$-complex, and let $\phi: X \to \BS^3$ be a \lfl\ embedding. Then \ti\ a locally $2$-connected $2$-complex $X'$ containing $X$ as a topological subspace, and a \lfl\ embedding $\phi': X' \to \BS^3$, \st\ $\phi'(X)= \phi(X)$.

Moreover, any action $\Gamma \act X$  \wrt\ which $\sigma(\phi)$ is invariant extends to an action $\Gamma \act X'$  \wrt\ which $\sigma(\phi')$ is invariant.

Furthermore, $\pi_1(X')\isom \pi_1(X)$.
\end{lemma}

Before giving the formal proof, let us explain the intuition by going one dimension down. Recall that a \defi{plane graph} is a 1-complex embedded in $\BS^2$ or $\R^2$. Given a connected plane graph \g that has some cut-vertices, it is easy to extend \g into a plane super-graph $G^\otimes$ that is 2-connected by `fattening' it, i.e.\ adding new vertices and edges near each face-boundary; see \fig{figfat} and \Dr{otimes} for details. To prove \Lr{lem 2-con}, we will add new 2-cells to $X$ to `fatten' it in such a way that the effect on each of its link graphs will be the same as the above modification of \g into $G^\otimes$:

\begin{definition} \label{otimes}
Let $G$ be a finite, connected, plane graph. Let $G''$ be a plane multigraph obtained from $G$ by adding two  parallel edges $e',e''$ to each edge $e\in E(G)$, and embedding them so that the circle $e' \cup e''$ separates $e$ from the rest of $G$. Then, for each  $e$ of $G''$ with end-vertices $u,v$, subdivide $e$ into a path of length 3 by placing two new vertices $e_u,e_v$ inside $e$. Finally, for each vertex $v$ of $G''$, and each two edges $e,f$ incident with $v$ that appear consecutively in the plane, add an edge between $e_v$ and $f_v$. We embed these edges in such a way that they form a circle separating $v$ from any other vertex of $G$. Let $G^\otimes$ denote the resulting plane graph; see \fig{figfat}. 
\end{definition}

\begin{figure} 
\begin{center}
\includegraphics[width=0.9\linewidth]{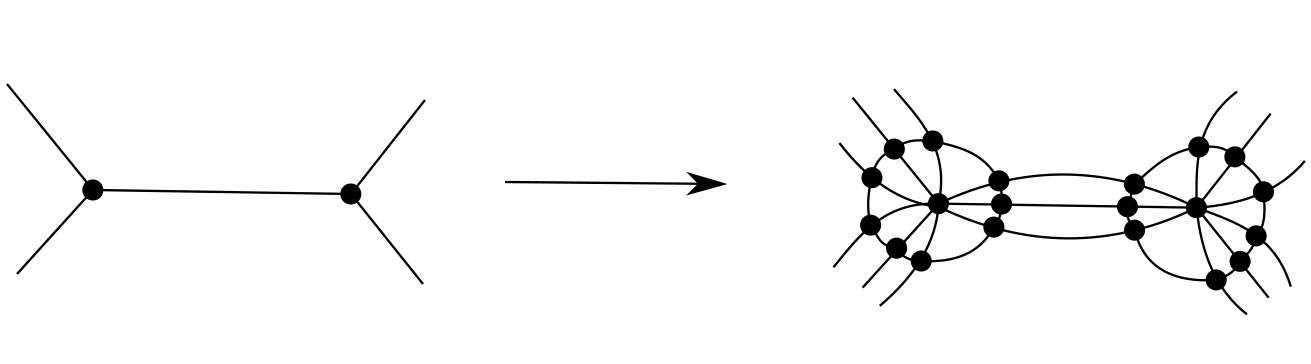} 
\end{center}
\caption{A portion of a graph $G$ (left half), and the corresponding part of $G^\otimes$ (right half).} \label{figfat}
\end{figure}

We remark that $G$ is a topological minor of $G^\otimes$.

\begin{lemma}\label{lem G 2-con}
Let $G$ be a connected, plane graph. Then $G^\otimes$  is $2$-connected.
\end{lemma}
\begin{proof}
This is straightforward, and boils down to checking that neither an original vertex of \G, nor one of the new vertices, can be a cut-vertex of $G^\otimes$.
\end{proof}

Using this we can now prove \Lr{lem 2-con}.
\begin{proof}[Proof of \Lr{lem 2-con}]
We may assume that $X$ is regular, for otherwise we can work with its barycentric subdivision, which preserves all assumptions we have made on $X$ as well as $\pi_1(X)$.

We begin the construction of $X'$ by `tripling' each 2-cell of $X$ as follows. For every $f\in X^2$, we introduce two new 2-cells $f^-,f^+$ with the same boundary and attaching map as $f$, and embed $f^-,f^+$ in $\BS^3$ \lfl\ and in such a way that their images bound a 3-ball that contains $f$ and is otherwise disjoint from $X$. Let $X^\pm$ be the resulting $2$-complex, and $\phi^\pm: X^\pm \to \BS^3$ the resulting \lfl\ embedding.

Next, we modify $X^\pm$ into $X'$ by engulfing each 1-cell $e\in X^1$ inside a copy of $\BS^2$. To make this more precise, we pick a \lfl\ homeomorph $S_e$ of $\BS^2$ in $\BS^3$, \st\ $S_e$ intersects $X^0$ at the endpoints of $e$, it intersects each 2-cell $f$ incident with $e$ along an arc, which we denote by $f_e$, and $S_e$ is otherwise disjoint from $X^\pm$ and all other $S_{e'}, e'\neq e$. It is easy to find such an $S_e$ inside a small neighbourhood of $e$. 

To turn the resulting subspace of $\BS^3$ into a 2-complex, we declare $f_e$ to be an 1-cell \fe\ pair $e,f$ as above, and we replace $f$ by the two 2-cells into which $f_e$ dissects it (one of which  2-cells will be further dissected by the other edges incident with $f$). 
Moreover, for every two 2-cells $f,g$ that are consecutive in the rotation system of $e$, the 1-cells  $f_e, g_e$ cut a `slice' of $S_e$, which we also declare to be a 2-cell of $X'$, and denote it by $s_{feg}$. This completes the construction of $X'$, and the PL embedding $\phi'$. As $X$ is a subspace of $X'$, we have $\phi'(X)= \phi(X)$. Notice that $X'$ has the same vertex set as $X$, and so to check that $X'$ is locally $2$-connected we just need to consider the effect of the newly added cells to each link graph $L_X(v), v \in X^0$. It is straightforward to check that the new link graph $L_{X'}(v)$ coincides with  $L_X(v)^\otimes$ as in \Dr{otimes} below. Thus $X'$  is locally $2$-connected by \Lr{lem G 2-con} below.

\medskip
For the second statement, we first extend the action $\Gamma \act X$ to $X^\pm$ as follows. \Fe\ $\gamma\in \Gamma$ and $f\in X^2$, we let $\gamma$  map the new 2-cells  $f^-,f^+$ bijectively to $(\gamma f)^-,(\gamma f)^+$. There are two ways to do so, and we choose the unique option that retains the invariance of the rotation system $\sigma:=\sigma(\phi)$, i.e.\ the choice that ensures that $\gamma \sigma_d = (-1)^{\eta(\gamma)} \sigma_{\gamma d}$ for some, hence every, directed edge $d$ incident with $f$, where $\eta: \Gamma \to \Z_2$ is a homomorphism as in \eqref{def cov}, witnessing the fact that $\sigma$ is  invariant. This ensures that the rotation system of $\phi^\pm$ is invariant \wrt\ the resulting action $\Gamma \act X^\pm$. Extending further to $\Gamma \act X'$ is straightforward: we just let $\gamma\in \Gamma$ map each $s_{feg}$ to $s_{(\gamma f)(\gamma e)(\gamma g)}$.

\medskip
Finally, it is easy to prove $\pi_1(X')\isom \pi_1(X^\pm)\isom \pi_1(X)$ by applying \vKT\ whenever a new 2-cell is introduced, using the fact that every new 2-cell forms a copy of $\BS^2$ with existing 2-cells.
\end{proof}

{\bf Remark:} \Lr{lem G 2-con} would remain true if instead of tripling each edge we just doubled it. The reason we triple is that in \Lr{lem 2-con} we have to triple each 2-cell in order to maintain the invariance of the action.

\medskip
Next, we observe that we can increase the local connectivity further from 2 to 3, using a construction of \cite{Whitney3D}. It was shown in \cite[\S 6]{Whitney3D} how given a locally 2-connected, simplicial, $2$-complex $X$, and an embedding $\phi: X \to \BS^3$, one can construct a super-complex ${\rm fat}(X)={\rm fat}(X,\phi)$ with improved properties:

\begin{lemma}[{\cite[Lemma~6.3]{Whitney3D}}] \label{lem 3-con}
Suppose that $X$ is a locally 2-connected, simplicial, $2$-complex,  and $\phi: X \to \BS^3$ is a locally flat embedding. Then ${\rm fat}(X)$ is locally $3$-connected. 
\end{lemma}

Moreover,  $\phi$ extends to a \lfl\ embedding $\phi': {\rm fat}(X) \to \BS^3$ (\cite[Lemma~6.1]{Whitney3D}). As the construction of  ${\rm fat}(X)$ is canonical, any group action $\Gamma \act X$ extends to an action $\Gamma \act {\rm fat}(X)$. 
Finally, any loop in ${\rm fat}(X)$ is homotopic to a loop in $X$ by the construction, and so ${\rm fat}(X)$ is simply connected if $X$ is.

\medskip
We now have all the ingredients needed for the main result of this section:

\begin{proof}[Proof of \Tr{complex to action}]
By Remark~\ref{rem lfl}, we may assume \obda\ that $\phi$ is \lfl. By \Lr{lem 1-con}, $X$ satisfies the conditions of \Lr{lem 2-con}, and we let $X'$ be the locally 2-connected 2-complex provided by the latter, and $\phi': X' \to \BS^3$ the corresponding \lfl\ embedding (the case where $X$ has fewer than 3 vertices is trivial). Since the rotation system of $\phi$ is $\Gamma$-invariant by assumption, the second sentence of \Lr{lem 2-con} yields an action $\Gamma \act X'$ \wrt\ which $\sigma(\phi')$ is $\Gamma$-invariant. By the third sentence of \Lr{lem 2-con} $X'$ is simply connected since $X$ is. By applying a barycentric subdivision (twice) if needed, we may assume that $X'$ is in addition a simplicial complex.

Next, we consider $Y:= {\rm fat}(X')={\rm fat}(X',\phi')$, which is locally 3-connected by \Lr{lem 3-con}. By the remarks following \Lr{lem 3-con}, we also obtain a \lfl\ embedding $\chi: Y \to \BS^3$, and an extension $\Gamma \act Y$ of the above action. 
Moreover, $Y$ is still simply connected. 

We finish by applying \Tr{thm Whitney3D} to $Y$, and pairs of embeddings of the form $\chi, \chi \circ a$ for each $a\in \Gamma$.  To make this precise, we recall that $\Gamma$ acts on $\chi(Y)\subset  \BS^3$, and we want to extend each $a\in \Gamma$ into a homeomorphism $h_a: \BS^3 \to \BS^3$. We let $h_a$ be the homeomorphism $\alpha$ obtained from \Tr{thm Whitney3D} when applied to the two embeddings $\chi$ and $\psi:= \chi \circ a$ of $Y$. In order for this map $a \mapsto h_a$ to be an action on $\BS^3$, we need it to be a homomorphism from $\Gamma$ to $Aut(\BS^3)$. This will not be the case in general if we let  \Tr{thm Whitney3D}  output any $h_a$ satisfying $\psi= h_a \circ \chi$, because for example $h_{a^{-1}}$ may differ from $(h_a)^{-1}$. But we can control the output of \Tr{thm Whitney3D}  by exploiting its second statement. This ensures that $g \mapsto h_a$ is a homomorphism from $\Gamma$ to $Aut(\BS^3)$ as desired, because restricting each $h_a$ to $Y$ recovers the action of $\Gamma$ on $Y$, which is a homomorphism. (This idea is spelt out in more detail in \cite[Lemma 5.6]{Kleinian}.)

Notice that $h_a \circ \chi = \psi = \chi \circ a$, i.e.\ $\chi(Y)$ is invariant \wrt\ the action we just defined.
\end{proof}

\section{From Group Actions to Invariant Embedded \gcc es} \label{sec i to iv} 

In this section we prove the implication \ref{m i} $\to$  \ref{m iv} 
of \Tr{main thm}. An \defi{embedded} 2-complex in a 3-manifold $M$ is a homeomorphic image of a 2-complex in $M$. Given an action $\Gamma \act M$, we say that an embedded 2-complex $X$ is \defi{$\Gamma$-invariant} if $\Gamma \act M$ preserves $X$ setwise.

\begin{theorem} \label{i to ii}
Let $\Gamma$ be a finite group, and $\Gamma \act \BS^3$  a faithful  action by homeomorphisms. Then $\Gamma$ admits an embedded, $\Gamma$-invariant, \gcc. 
\end{theorem}

In fact we will prove the following more general statement, which yields a generalisation of the implication \ref{m i} $\to$  \ref{m iv} of \Tr{inf thm}. We say that a chamber $C$ of an embedded 2-complex $Y\subset M$ is \defi{finitary}, if $\partial C$ is a finite subcomplex of $Y$.

\begin{theorem} \label{cor i to ii}
Let $M$ be a 3-manifold, let $\Gamma$ be a finitely  \mymargin{} generated group, and $\Gamma \act M$  a faithful, \pd, co-compact  action by homeomorphisms. Then there is a $\Gamma$-invariant embedded 2-complex $Y\subset M$ with finitary chambers \st\ $\Gamma$ acts regularly on $X^0$, and $\pi_1(Y)\isom \pi_1(M)$.

In particular, when $M$ is \sico, then $\Gamma$ admits a $\Gamma$-invariant  \gcc\ embedded in $M$. 
\end{theorem}

Notice that the statement that $Y$ has finitary chambers implies in particular that the vertices of $Y$ have no accumulation point in $M$. \new

\medskip
Recall that $2$-complex $C$ is a generalised Cayley complex of  $\Gamma$ if $C$ is simply connected and $\Gamma$ admits an action on $C$ that is regular on $C^0$. We will construct such a complex embedded in $M$ in two steps. In the first step we construct an embedded 
$2$-complex $X \subset M$ \st\ our action $\Gamma \act M$ is regular on the chambers of $X$ (\Lr{action to emb com}). In the second step we perform local modifications on $X$ to transform regularity on the chambers into regularity of the action on the vertices.

\subsection{Step 1: Constructing an embedded 
$2$-complex with a regular action on its chambers}

The following lemma performs the first step of our construction of a \gcc\ of $\Gamma$ as mentioned above:

\begin{lemma} \label{action to emb com}
Let $M$ be a topological 3-manifold, and let  $\Gamma \act M$ be a faithful,  \pd, co-compact group action by homeomorphisms. Then there is a $\Gamma$-invariant, embedded 2-complex $X\subset M$, \st\ $\pi_1(X)\isom \pi_1(M)$, and $\Gamma$ acts regularly on the chambers of $X$, each of which is finitary and homeomorphic to $\R^3$.
\end{lemma}

For the proof of this we will use the following basic fact: 
\comment{
	\begin{lemma} \label{lem S2 chambers}
Let $M$ be a topological 3-manifold, and let  $X\subset M$ be an embedded 2-complex,  \st\ each chamber of $X$ is precompact and homeomorphic to $\R^3$. 
Then $\pi_1(X)\isom \pi_1(M)$. \qed
	\end{lemma}
}

\begin{lemma} \label{lem pi1 chambers}
Let $M$ be a topological 3-manifold, and let  $X\subset M$ be an embedded 2-complex \st\ each chamber of $X$ is homeomorphic to $\R^3$, and it is bounded by a finite subcomplex of $X$. Let $f\in X^2$ be a 2-cell contained in the boundary of two distinct chambers.  
Then $\pi_1(X)\isom \pi_1(X - f)$.
\end{lemma}
\begin{proof}
Let $C_1,C_2$ denote the two chambers having $f$ in their boundaries. We can continuously deform $f$ via $C_1$ (or $C_2$) onto a continuous image $f'\subset \partial C_1 - f$ of a topological disc using the fact that $C_1$ is homeomorphic to  a ball in $\R^3$, and $f\subset \partial C_1$ is homeomorphic to a disc. We can use $f'$ to show that the circle $\partial f$ is 0-homotopic in $X - f$. 
Thus we have $\pi_1(X)\isom \pi_1(X - f)$ by \vKT\ since $f$ is \sico.
\end{proof}

\begin{proof}[Proof of \Lr{action to emb com}]
We may assume \obda\ that our action $\Gamma \act M$ is smooth by Pardon's \Tr{pardon new}. 

It is known that for every such action, the quotient space $M/\Gamma$ ---which is a 3-orbifold, but the reader will not need to know what this means--- admits a triangulation $T$ \cite[Proposition~1.2.1]{MoPrSim}, which is \defi{adapted} to the action in the sense that for each simplex $\sigma$ of $T$, the stabilisers under $\Gamma$ of all pre-images of points in $\sigma$ are isomorphic to each other. Since $\Gamma \act M$ is co-compact, $M/\Gamma$ is compact, and thus $T$ is finite. Let $\pi: M \to M/\Gamma$ be the quotient projection. Its inverse $\pi^{-1}$ lifts $T$ to a triangulation $\tilde{T}$ of $M$, as proved in \cite[Lemma~1.2.2]{MoPrSim}, which is $\Gamma$-invariant by construction. We think of the 2-skeleton $\tilde{T}^2$ of $\tilde{T}$ as an embedded 2-complex in $M$. It is straightforward to check that the chambers of $\tilde{T}$ are exactly its 3-cells. It is easy to show that $\pi_1(\tilde{T}^2)\isom \pi_1(M)$ by applying \vKT\ to the 3-cells of $\tilde{T}$.

Notice that the action $\Gamma \act M$ is free on the 3-cells of $\tilde{T}$, and therefore on the chambers of $\tilde{T}^2$, because it is faithful. Indeed, if an element $g$ of $\Gamma$ fixed a 3-cell $C$ setwise, then $g$ would have to fix $C$ pointwise. This would force $g$ to also fix the 3-cells incident with $C$, hence all of $M$ by its connectedness, implying that $g$ can only be the identity of $\Gamma$. 

If the action is not transitive on the chambers of $\tilde{T}^2$, then we can find a subcomplex $X$ of $\tilde{T}^2$ which maintains the other desired properties and such that $\Gamma$ acts transitively on the chambers of $X$, by finding an appropriate fundamental domain of 3-cells of $\tilde{T}^2$ and joining them into one chamber. To do so we introduce the following notion.

Say that two chamber-boundaries $C,D$ of a 2-complex $H$ are \defi{adjacent}, if their boundaries share a 2-cell of $H$. Say that $C,D$ are \defi{tight-connected}, if there is a sequence $C_1,\ldots C_k$ of chamber-boundaries \st\ $C_1=C, C_k=D$, and $C_i$ is adjacent to $C_{i+1}$ \fe\ $1\leq i<k$. A \defi{tight-component} is a maximal tight-connected set of chamber-boundaries of $H$. It is straightforward to check that 
\labtequ{tight}{the boundaries of 3-cells of a triangulation of any connected 3-manifold form a single tight-component.}

Let $F$ be a maximal tight-connected set of boundaries of 3-cells of $\tilde{T}^2$ that contains at most one representative from each $\Gamma$-orbit of 3-cells of $\tilde{T}$. The maximality of $F$, combined with \eqref{tight}, easily implies that $F$ contains a representative from each $\Gamma$-orbit of 3-cells, for otherwise we could add to $F$ a 3-cell adjacent with one of its elements (this is a well-known idea, appearing e.g.\ in \cite{BabaiContr}). Thus $F$ contains exactly one representative of each $\Gamma$-orbit of 3-cells, in other words, the action of $\Gamma$ on the translates of $F$ is regular. Moreover, $F$ is finite since $\Gamma \act M$ is co-compact.

We claim that there is a set $D$ of 2-cells of $\bigcup F$ \st\ $\bigcup F \sm D$ has only one chamber $C$, and moreover $C$ is homeomorphic to $\BS^3$. Indeed, we can construct $D$ recursively as follows. As long as $\bigcup F$ has more than one chamber (each homeomorphic to $\R^3$) we can find two of them $C_1,C_2$  sharing a 2-cell $f$ by tight-connectedness. By removing $f$ we join $C_1,C_2$ into one chamber, which is homeomorphic to $\R^3$ since both $C_1,C_2$ are. (The boundary of the new chamber need not be homeomorphic to $\BS^2$, however.) It is easy to see that the tight-connectedness of the chamber-boundaries of $F$ is preserved. Since $F$ is finite, this recursion terminates leaving a single chamber, proving our claim.

Let $X\subset M$ be the 2-complex obtained from $\tilde{T}^2$ by removing a set $D$ as above along with all its $\Gamma$-translates. Then $\bigcup F$ is contained in one chamber of $X$, and it follows that $\Gamma$ acts regularly on the chambers of $X$. 

By construction, each chamber of $X$ is still homeomorphic to $\R^3$ and finitary. 

Notice that whenever we removed a 2-cell $f$ of $\tilde{T}^2$ we joined two chambers $C_1,C_2$ into one, and so we did not change $\pi_1$ by \Lr{lem pi1 chambers}. Thus $\pi_1(X)\isom \pi_1(\tilde{T}^2)$, which coincides with  $\pi_1(M)$ as noticed above.
\end{proof}

\subsection{Step 2: From regularity on the chambers to  regularity on the vertices}

Having constructed an embedded 2-complex $X\subset M$ \st\ the action $\Gamma \act M$ of \Tr{i to ii} or \ref{cor i to ii} is regular on the chambers of $X$, our aim now is to modify $X$ locally so that the action becomes regular on the vertices. Most of the work will go into making the action free on the vertices, because having done so we will be able to use the standard trick of contracting a fundamental domain to achieve transitivity. To formulate this trick in our setup, given a graph $G$, and a subgroup $\Gamma$ of the automorphism group $Aut(G)$ of $G$, we call a  subgraph $H\subseteq G$ a \defi{fundamental domain} for $\Gamma$, if $H$ contains exactly one vertex from each $\Gamma$-orbit. 

\begin{lemma}[Babai's Contraction Lemma \cite{BabaiContr}] \label{lem Babai}
Let \g be a connected graph, and suppose a group $\Gamma\leq Aut(G)$ acts freely on the vertex set $V(G)$. Then there is a connected subgraph $D \subset G$ that is a fundamental domain for the action, and the graph $G / D$  obtained by contracting each $\Gamma$-image of $D$ into a point is a generalised Cayley graph of $\Gamma$. In particular, $\Gamma$ acts transitively on $V(G / D)$.
\end{lemma}

The graph we will later apply \Lr{lem Babai} to is the 1-skeleton of our 2-complex.

\medskip
Thus it remains to transform the freeness of the action on the chambers arising from \Lr{action to emb com} into freeness on the vertices. This is carried out by the following result.

\comment{
	\begin{theorem} \label{free the action}
Let $\Gamma \act \BS^3$ be a  finite group action by homeomorphisms, and let $X\subset \BS^3$ be a $\Gamma$-invariant,  \sico,  embedded 2-complex, each chamber of which is finitary and homeomorphic to $\R^3$. Suppose moreover that the action of $\Gamma$ on the chambers of $X$ is regular. Then there is a $\Gamma$-invariant, \sico, embedded 2-complex  $F(X) \subset \BS^3$, \st\ the action of $\Gamma $ on $V(F(X))$ is free. \mymargin{\tiny Add that `$F(X)$ can be chosen to be \pl\ if $X$ is' if needed.}
\end{theorem}
}

\begin{theorem} \label{free the action}
Let $M$ be a topological 3-manifold, and let $\Gamma \act M$ be properly discontinuous group action by homeomorphisms. Let $X\subset M$ be a $\Gamma$-invariant,  embedded 2-complex, \st\ $\pi_1(X)\isom \pi_1(M)$, the action of $\Gamma$ on the chambers of $X$ is regular, and each chamber is finitary and homeomorphic to $\R^3$. Then there is a $\Gamma$-invariant, embedded 2-complex  $F(X)\subset M$ with finitary chambers, \st\ the action of $\Gamma $ on $V(F(X))$ is free, and $\pi_1(F(X))\isom \pi_1(M)$.
\end{theorem}

Our formal definition of the complex $F(X)$ that achieves this takes some time, but the idea is rather simple: given a vertex $x$ of $X$ stabilised by $\Gamma \act M$, we notice that we can pick a set $O_x$ of nearby points  inside the chambers incident with $x$ \st\ the action of $\Gamma$ on $O_x$ is regular, because $\Gamma$ acts freely on the chambers. The idea is to blow up each 2-cell, 1-cell, and 0-cell of $X$ into a homeomorph of $\BS^2$, in order to modify $X$ into a complex $F(X)$ with vertex set $\bigcup_{x\in X^0} O_x$. An example is shown in \fig{fig cuboct}: if $X$ is the standard cubic lattice embedded in $\R^3$ (top left), then a portion of $F(X)$ is displayed in the bottom right of the figure.

The reader will lose nothing by assuming that $M= \BS^3$ and $\Gamma$ is finite throughout this section; this is enough for proving \Tr{main thm}, and this assumption makes no difference for any the proofs in this section. 
\medskip


\begin{figure} 
\begin{center}
\begin{minipage}{.5 \linewidth}
\begin{tabular}{ll}
\includegraphics[width=0.55\linewidth]{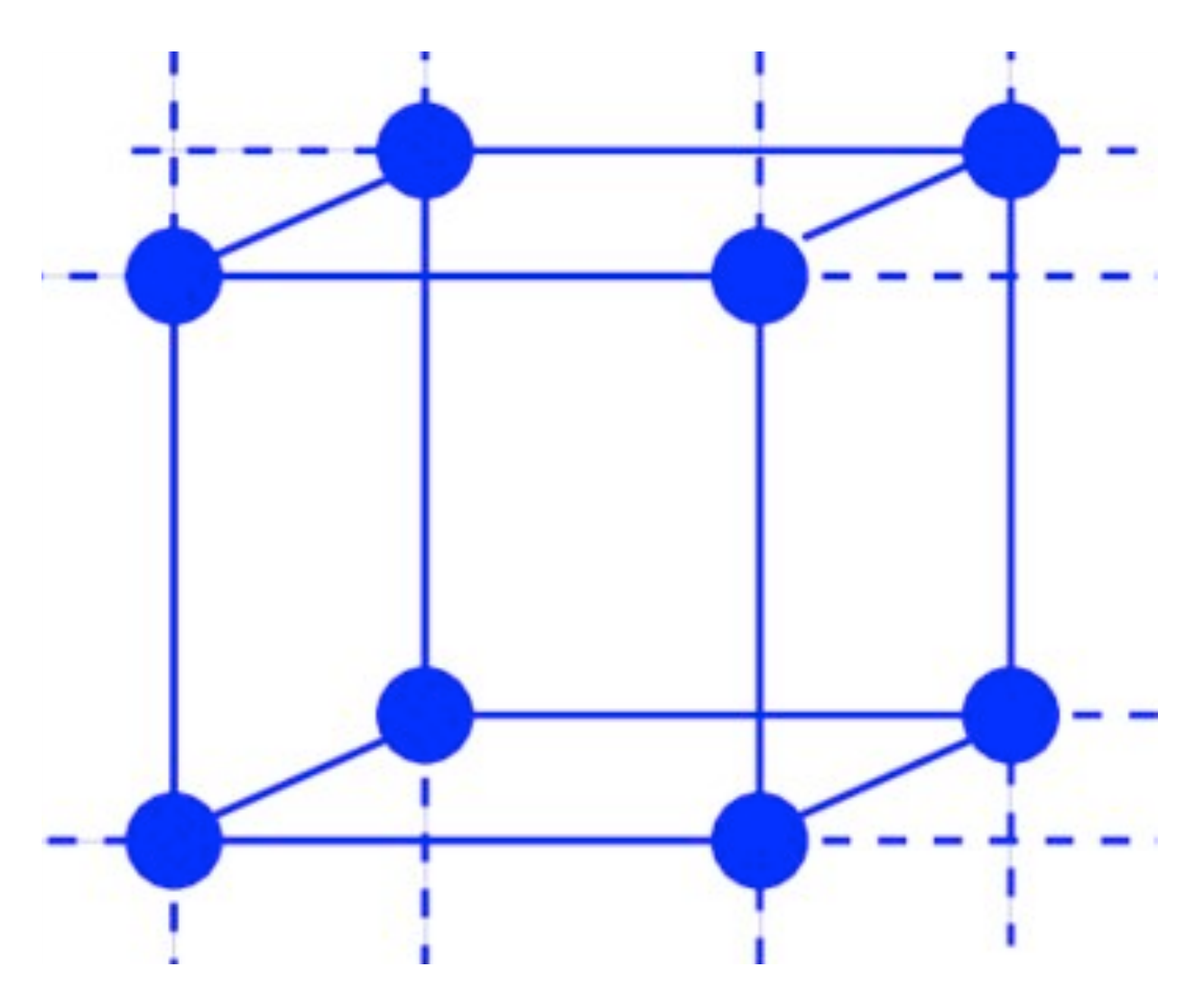} & \reflectbox{\includegraphics[width=0.45\linewidth,angle=16,origin=c]{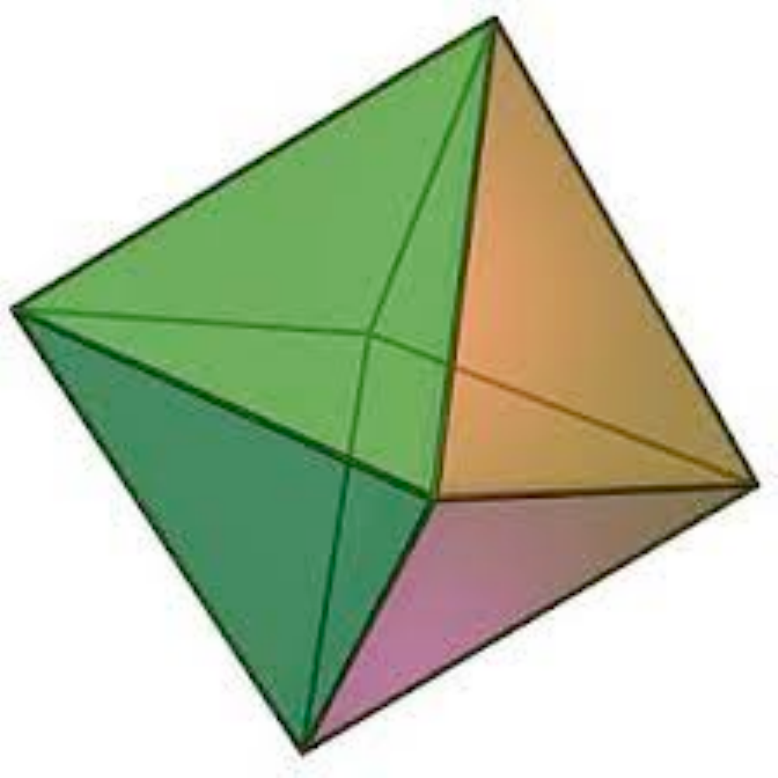}}    \\
\includegraphics[width=0.45\linewidth]{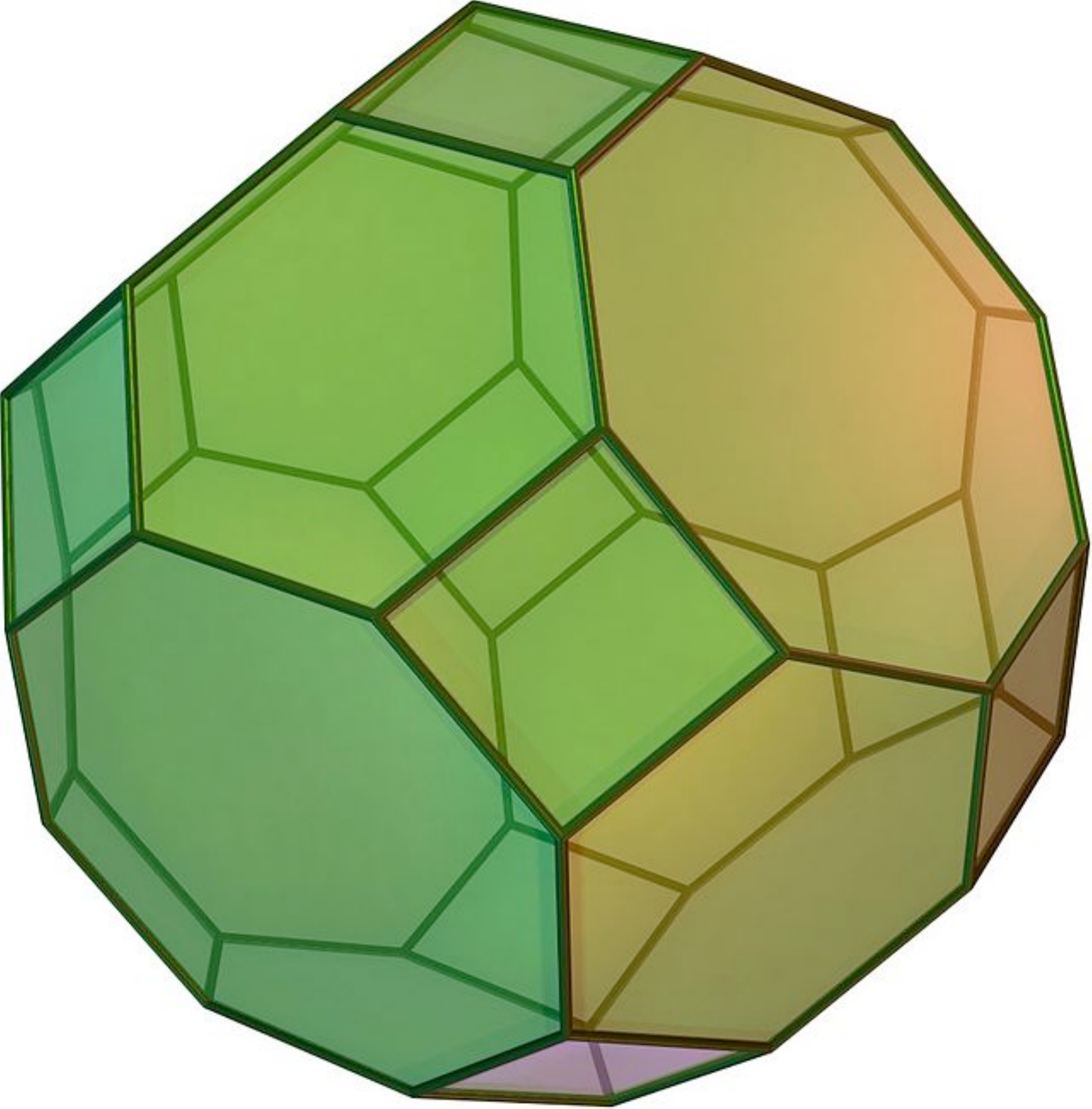} &  \includegraphics[width=0.45\linewidth]{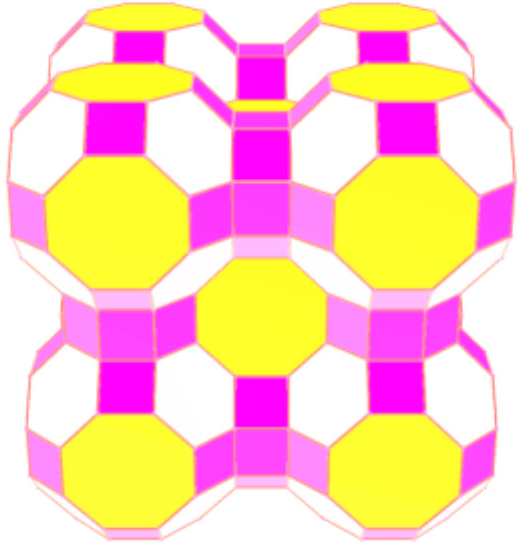}    \\
\end{tabular}
\end{minipage}
\end{center}
\caption{An example $F(X)$, when $X$ is the cubic lattice (top left). Each link graph is an octahedron (top right), and so each pineapple of $F(X)$ is a truncated cuboctahedron (bottom left). The pineapples are arranged as in the bottom right figure, which shows four of them in the front.} \label{fig cuboct}
\end{figure}

We now prepare for the formal definition of $F(X)$.
Given an embedded 2-complex $X\subset M$, and a homeomorphic image $S$ of $\BS^2$ in $M$ which is locally flat, we say that $S$ is \defi{adapted} to $X$ if 
\begin{enumerate}
\item \label{c i} for every 1-cell $e\in X^1$ the intersection $S \cap e$ is either a single point, or all of $e$, or empty;
\item \label{c ii} for every 2-cell  $f\in X^2$, the intersection $S \cap f$ is either an arc between two points of the boundary of $f$ (either 0-cells, or interior points of 1-cells), or empty, and
\item \label{c iii} $S$ separates  $M$ into two components.
\end{enumerate}
If $S$ is adapted to $X$, and $A$ is one of the two components into which $S$ separates  $M$, 
then we can obtain an embedded 2-complex $X_A$ from $X$ by removing $X \cap A$ and adding $S$ to $X$; to make this more precise, we define the \defi{$A$-truncation} of $X$ to be the embedded 2-complex $X_A$ obtained as follows. 
\begin{enumerate}
\item \label{t i} For every 1-cell $e\in X^1$ that intersects $S$ at a point $p$, we declare $p$ to be a 0-cell of $X_A$, we declare the subarc of $e$ lying outside $A$ to be a 1-cell of $X_A$, and discard the subarc of $e$ lying inside $A$. 
\item \label{c ii} For every 2-cell  $f\in X^2$ intersecting $S$ along an arc $P$ between two points $x,y$ of the boundary of $f$, we declare $P$ to be a 1-cell of $X_A$ ---its end-points $x,y$ must be 0-cells of $X_A$ by \ref{t i}. Notice that $f \sm (S \cup A)$ is homeomorphic to a disc $f'$, and we declare $f'$ to be a  2-cell of $X_A$, discarding $f$.
\item \label{c iii} The points $p$ and arcs $P$ as in \ref{c i}--\ref{c ii}subdivide $S$ into topological discs, which we declare to be 2-cells of $X_A$.
\item \label{c iv} For every cell $C$ of $X$ that does not intersect $S$, we keep $C$ in $X_A$ if it lies outside $A$, and discard it if it is contained in $A$.
\end{enumerate}

We now construct the 2-complex $F(X)$ featuring in \Tr{free the action} by a combination of such truncations.

\begin{definition} \label{fruit cx}
Given an embedded 2-complex $X\subset M$, we construct another embedded 2-complex $F(X)\subset M$ as follows. 
\end{definition}
\begin{enumerate}
\item \label{f i} We blow each 2-cell $f\in X^2$ up like a \defi{mango}; that is, we replace $f$ by two `parallel' copies $f',f''$ with the same boundary and attachment map, and embed $f',f''$ into $M$ so that one of the sides of the 2-sphere $S_f:= f' \cup f''$ contains $f$ and is otherwise disjoint from $X$.  Moreover, we ensure that $S_f$ is locally flat, and disjoint from $g' \cup g''$ \fe\ $g\neq f\in X^2$ except possibly for intersections along $X^1$; in other words, our mangos do not cross each other. Let $X_1$ denote the resulting 2-complex. We call $f' \cup f''$ the \defi{mango} of $f$, and imagine its side containing $f$ as one.

\item \label{f ii} Next, we blow each 1-cell $e\in X_1^1$ up like a \defi{banana}. To define this formally, let $S_e$ be a homeomorph of $\BS^2$ in $M$ \st\ the end-vertices of $e$ lie on $S_e$, one side $A_e$ of $S_e$ contains the interior of $e$ and is otherwise disjoint from the 1-skeleton of $X_1$, and $S_e$ is adapted to $X_1$ and locally flat. It is easy to find such $S_e$, and to ensure that they are pairwise disjoint except possibly at their vertices.
We apply the $A_e$-truncation of $X_1$ for each $e\in X_1^1$ to obtain a new embedded 2-complex  $X_2$. We call $S_e$ the \defi{banana} of $e$.

\item \label{f iii} Finally, we blow each vertex $v\in X_2^0=X^0$ up like a \defi{pineapple}; that is, we pick a homeomorph $S_v$ of $\BS^2$ in $M$, \st\ one side $A_v$ of $S_v$ contains $v$ but no other vertices of $X_2$, and $S_v$ is adapted to $X_2$ and locally flat. Moreover, we choose the $S_v, v\in X_2^0$ small enough that they are pairwise disjoint. 
We apply the $A_v$-truncation of $X_2$ for each $v\in X_2^0$ to obtain the desired 2-complex $F(X)$. We call $S_v$ the \defi{pineapple} of $v$.
\end{enumerate}

\begin{remark} \label{rem pine}
The 1-skeleton of the pineapple of $v$ can be obtained from the link graph of $v$ by doubling each edge by a parallel one, and then blowing up each vertex of the resulting plane graph into a cycle of  vertices of degree 3. 
\end{remark}

{\bf Example:} When $X$ is the cubic lattice in $\R^3$, the link graph of each vertex is isomorphic to the 1-skeleton of the octahedron. Each pineapple of $F(X)$ is a truncated cuboctahedron. They are arranged as shown in \fig{fig cuboct}.

\begin{remark} \label{rem flag}
There is an alternative, more abstract, way to define $F(X)$. A \defi{flag} of $X$ is a 4-tuple $(c_0,c_1,c_2,c_3)$ where $c_i$ is an $i$-cell of $X$, and $c_i$ is incident with $c_{i-1}$ for $1\leq i \leq 4$, with the convention that the 3-cells of $X\subset M$ are its chambers. We can identify the set of 0-cells of $F(X)$ with the set of flags of $X$. We connect two flags with an 1-cell of $F(X)$ whenever they differ in exactly one coordinate. We can 4-colour the 1-cells of $F(X)$ using the coordinate at which its end-vertices differ as a colour. The 2-cells of $F(X)$ are bounded by the 2-coloured cycles \wrt\ this colouring. Again \fig{fig cuboct} can serve as an example.  This definition generalises in any dimension. 

The letter $F$ in our notation $F(X)$ stands for `fruit', but also for `flag'.
\end{remark}

\medskip 

Using \vKT\ it will be easy to deduce that $F(X)$ preserves the fundamental group of $X$:
\begin{lemma} \label{lem pi}
For $F(X)$ as in \Dr{fruit cx}, we have $\pi_1(F(X))\isom \pi_1(X)$.
\end{lemma}
\begin{proof}
Notice that if we contract each pineapple $S_v$ in the construction of $F(X)$ to a point, we obtain a 2-complex homeomorphic to $X_2$. Thus $\pi_1(F(X))\isom \pi_1(X_2)$ by \vKT. Similarly, squeezing each banana $S_e$ in $X_2$ back to an edge with the same endpoints as $e$, results into a 2-complex homeomorphic to $X_1$, and so $\pi_1(X_2)\isom \pi_1(X_1)$. Finally, squashing each mango $S_f$ of $X_1$ onto a disc with the same boundary as $f$ results in a homeomorph of $X$, yielding  $\pi_1(X_1)\isom \pi_1(X)$.
\end{proof}

In order to be able to use $F(X)$ to prove \Tr{free the action}, we need to construct it more carefully so that $\Gamma \act M$ extends to an action on $F(X)$. This would be easy if $\Gamma$ acted freely on $X$, but in general we need to take some care to ensure that the stabiliser of each 2-cell, 1-cell or 0-cell fixes the corresponding mango, banana, or pineapple, respectively. We will be able to achieve this by choosing a chamber $C$ of $X$ and using its closure $\cls{C}$ as a fundamental domain. More precisely, we will prove

\begin{lemma} \label{lem FD}
Let $X$ and $\Gamma$ be as in the statement of  \Tr{free the action}, let $C$ be a chamber of $X$, and let $D$ be the sub-complex of $X$ bounding $C$. Then there is a homeomorphic copy $F\subset M$ of $F(D)$, \st\ $\Gamma (F \cap \cls{C})$ is homeomorphic to $F(X)$.
\end{lemma}
Here, $\Gamma A$ denotes the image of a set $A\subset M$ under the action $\Gamma \act M$, and $F(D)$ is given by \Dr{fruit cx}.

Before proving \Lr{lem FD}, let us see how it implies \Tr{free the action}.

\begin{proof}[Proof of \Tr{free the action}]
Given $X$ as in the statement, we construct $F(X)$ as in \Dr{fruit cx}. By \Lr{lem pi}, $F(X)$ is \sico\ since $X$ is. Easily, $F(X)$ has finitary chambers since $X$ does.

Let $D$ be the sub-complex of $X$ bounding a chamber $C$ of $X$. Then \Lr{lem FD} yields an embedded 2-complex $F\subset M$ \st\ $F':=\Gamma (F \cap \cls{C})$ is homeomorphic to $F(X)$. 
Notice that $F'$ is $\Gamma$-invariant by definition.

It remains to check that the action of $\Gamma$ on $V(F')$ is free. This is true because each vertex of $F'$ lies in the interior of a chamber of $X$, and $\Gamma$ acts freely on the chambers of $X$ by assumption.

\end{proof}

It remains to prove \Lr{lem FD}. To construct the desired copy $F$ of $F(D)$, we will first design the intersection of $F$ with each 2-cell of $F(D)$. To do so, we need to remember how the bananas and pineapples of $F(X)$ intersect each 2-cell of $X$ (the mangos do not); these intersections are described in the following definition, but they are easier to see in \fig{fig sl pat}. 

\begin{figure} \begin{center}
\includegraphics[width=0.35\linewidth]{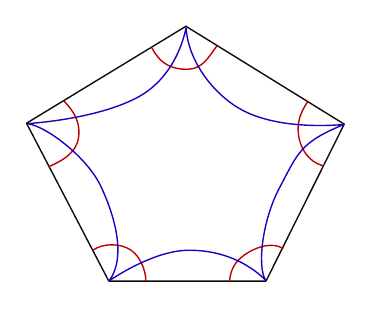} 
\end{center}
\caption{A topological 5-gon with a slice pattern. The $E_i$ are depicted in red (if colour is shown), and the $V_i$ in blue.} \label{fig sl pat} 
\end{figure}

A \defi{topological $n$-gon} is a regular 2-complex $P$ containing exactly one 2-cell $f$, and \st\ $P^1$ is homeomorphic to $\BS^1$ (and coincides with the boundary of $f$). 

\begin{definition} \label{def pattern}
Let $P$ be a \defi{topological $n$-gon}, with vertices $v_1,\ldots v_n$, and edges $v_i v_{i+1\pmod n}$. A \defi{slice pattern} on $P$ consists of two sets $\ce:= \{E_1, \ldots, E_n\}$ and $\cv:= \{V_1, \ldots, V_n\}$ of arcs on $P$, \st\ (\fig{fig sl pat})
\begin{enumerate}
\item \label{sp i} the end-points of $E_i$ are $v_i$ and $v_{i+1\pmod n}$;
\item \label{sp ii} the end-points of $V_i$ are interior points of the  edges $v_{i-1\pmod n} v_i$ and $v_i v_{i+1\pmod n}$;
\item \label{sp iii} each $E_i$ and $V_i$ meets the 1-skeleton of $P$ at its end-points only;
\item \label{sp iv} the elements of $\ce$ are pairwise disjoint, and so are the elements of $\cv$, and
\item \label{sp v} $V_i$ is disjoint from $E_j$ unless $j=i$ or $j=i-1\pmod n$.
\end{enumerate}
\end{definition}

An \defi{automorphism} of a topological $n$-gon $P$ is a homeomorphism of $P$ mapping each vertex to a vertex (and hence each edge to an edge). To prove  \Lr{lem FD} we will apply the following lemma to each 2-cell of $D$; this helps us by pushing the difficulty one dimension down.

\begin{lemma} \label{sl pat}
Let $P$ be a topological $n$-gon, and $h: P \to P$ an automorphism \st\ $h^2$ is the identity. Then there is a slice pattern $(\ce,\cv)$ of $P$ preserved by $h$. That is, $h$ maps each element of $\ce$ to an element of $\ce$, and each element of $\cv$ to an element of $\cv$.
\end{lemma}
\begin{proof}
Notice that $h$ must fix some arc $A$ joining two boundary points of $P$, and exchange the two components into which $A$ separates $P$. We can thus pick `half' a slice pattern on the quotient polygon $P/h$, and lift it back to $P$ to obtain a slice pattern of $P$. We have assumed here that $h$ in not the identity, in which case the statement is trivial.
\end{proof}

\begin{proof}[Proof of \Lr{lem FD}]
Notice that although $\Gamma\act M$ is free on the chambers of $X$, some of the 2-cells of $D$ may have a non-trivial stabiliser in $\Gamma$. However, for every such 2-cell $f\in D^2$, there is at most one non-identity element $h_f$ of $\Gamma$ fixing $f$, because $f$ is in the boundary of at most two chambers of $X$, and no non-identity element of $\Gamma$ fixes a chamber. For the same reason, $h_f$ must be an involution. Applying \Lr{sl pat} with $P=\cls{f}$ we obtain 
a slice pattern $(\ce,\cv)$ of $\cls{f}$ preserved by $h_f$. Choosing such a slice pattern for one representative $f$ of each $\Gamma$-orbit of 2-cells of $D$, and translating it to the other representatives via the action of $\Gamma$, we obtain a family $(\ce_f,\cv_f)_{f\in D^2}$ of slice patterns of all 2-cells of $D$, which family is \defi{compatible} with our action $\Gamma\act M$, i.e.\ $(\ce_f,\cv_f)=(\ce_{gf'},\cv_{gf'})$ whenever $gf'=f$ for some $f,f'\in D^2$ and some $g\in \Gamma$. Notice that this implies that the set of $\Gamma$-translates of this family is $\Gamma$-invariant.

To find the desired copy $F$ of $F(D)$, we can start by picking the mangoes arbitrarily as in \Dr{fruit cx} \ref{f i}. (The half-mango outside $\cls{C}$ will be irrelevant.) For each $e\in D^1$, pick the corresponding banana $S_e$ (\Dr{fruit cx} \ref{f ii}) so that its intersection with each 2-cell of $D$ is contained in one of the $\ce_f$ of the above family of slice patterns. Similarly, for each $v\in D^0$, pick the corresponding pineapple $S_v$ (\Dr{fruit cx} \ref{f iii}) so that its intersection with each 2-cell of $D$ is contained in one of the $\cv_f$. This completes the construction of $F$.

It remains to check that $\Gamma (F \cap \cls{C})$ is homeomorphic to $F(X)$ as claimed. To see this, notice that  $F \cap \cls{C}$ contains one half of each mango, a slice of each banana, and a sector of each pineapple of $F$. Moreover, when acted upon by $\Gamma$, these portions combine well to produce a homeomorph of $F(X)$. Indeed, 
for any $g,h\in \Gamma$, the translates $g(F \cap \cls{C})$ and $h(F \cap \cls{C})$ are disjoint except possibly at the boundaries of $g \cls{C}$ and $h \cls{C}$, where they meet along the $\Gamma$-invariant family of slice patterns 
chosen above. 

\end{proof}

\comment{
	\begin{corollary} \label{cor regular}
In \Cr{cor free the action}, we may in addition assume that the action of $\Gamma $ on $V(F(X))$ is regular. \mymargin{perhaps we will join this into \Cr{cor free the action}.}
\end{corollary}
	\begin{proof}

	\end{proof}
}

We are now ready to complete the main result of this section.

\begin{proof}[Proof of Theorems \ref{i to ii} and \ref{cor i to ii}]
Apply \Lr{action to emb com} to obtain an embedded 2-complex $X\subset M$ \st\ $\pi_1(X)\isom \pi_1(M)$ and $X$ satisfies all the assumptions of \Tr{free the action}. 
Then apply \Tr{free the action} to this $X$, 
to turn it into an embedded 2-complex $F(X)$ \st\  the action of $\Gamma$ on $V(F(X))$ is free. 

Suppose that  the action of $\Gamma$ on $V(F(X))$ is not transitive. Let $D\subset F(X)^1$ be a connected fundamental domain for the action as provided by \Lr{lem Babai}. Notice that $D$ is finite because $\Gamma \act M$ is co-compact. Let $T$ be a spanning tree of $D$, and let $U$ be an open neighbourhood of $T$ in $M$ homeomorphic to $\R^3$.  Easily, we can pick $U$ small enough that its $\Gamma$-translates are pairwise disjoint. Contracting each $\Gamma$-translate of $U$ into a point, we obtain a manifold homeomorphic to $M$, still acted upon by $\Gamma$, into which manifold the 2-complex $Y:= F(X)/T$ is embedded $\Gamma$-invariantly. These contractions preserve the property that all chambers are finitary. By the choice of $T$, the action of $\Gamma$ on $V(Y)$ is regular. Moreover, $\pi_1(Y) \isom \pi_1(F(X)) \isom \pi_1(X)\isom \pi_1(M)$. In particular, when $M$ is \sico, then $Y$ is a \gcc\ of $\Gamma$.
\end{proof}

Notice that when $M$ is \sico\ in \Tr{cor i to ii}, then it is a special 3-manifold by \Tr{Agol}. Thus we have proved the implication  $\ref{m i}\rightarrow \ref{m iv}$ of Theorems~\ref{main thm} and \ref{inf thm}.

\begin{remark}
Since $X$, and hence $F(X)$, are regular 2-complexes by construction, the 2-complex $Y$ that the above proof provides is edge-regular (as defined in \Sr{ccs}). We do not know if we can always obtain a regular $Y$ in Theorems \ref{i to ii} and \ref{cor i to ii}.
\end{remark}

\section{Concluding remarks} \label{conc}

Our proofs of Theorems~\ref{main thm} and \ref{inf thm} are now complete. The implication $\ref{m iv}\rightarrow \ref{m i}$ is trivial because every special 3-manifold is \sico. The implications $\ref{m iv}\rightarrow \ref{m iii} \rightarrow \ref{m ii}$ are trivial because every special 3-manifold embeds in $\BS^3$. The implications \ref{inf ii} $\to$ \ref{inf iv}  and 
 $\ref{m i}\rightarrow \ref{m iv}$ have been proved in Sections~\ref{sec ii to iv} and \ref{sec i to iv}, respectively.


\bigskip

As mentioned in the introduction, there are groups admitting a   \Cc\ embeddable in $\R^3$ with invariant  \prs, but only if we allow infinite pre-chambers. Examples include $\Z^2$, and more generally any fundamental group of an orientable closed surface \cite{SEsurfacegroups}.

\begin{question} \label{Q prs}
Which infinite groups admit a Cayley complex embeddable in $\R^3$ with invariant \prs?
\end{question}

This class of groups contains the fundamental groups of  closed 3-manifolds (as proved by \Tr{main thm}) and closed surfaces; more generally, it contains all Kleinian function groups \cite[\S 12]{Kleinian}\footnote{This is not stated explicitly in \cite{Kleinian}, but the proof of Theorem 1.3 there constructs an embedding of a \Cc\ in $\R^3$.}. It is easy to see that it also  contains groups of the form $F \times \Z$ where $F$ is free. It would be interesting to clarify the relationship between the groups of Question~\ref{Q prs} and the Kleinian groups.

\begin{question} \label{Q list}
Is there a finite set $X$ of 3-manifolds, such that each of the  groups of Question~\ref{Q prs} admits a discrete action on an element of $X$?
\end{question}

This $X$ should contain the three special open 3-manifolds of \Tr{Agol}. It  should also contain  $\R^2 \times S^1$, because of $\Z^2$ and other surface groups. Moreover, $X$ should contain $C  \times \R$ where $C$ stands for the Cantor 2-sphere, i.e. $\BS^2$ with a Cantor set removed; we include $C  \times \R$ to let groups of the form $F \times \Z$ act. These 5 manifolds could suffice as far as we can tell.

\medskip
One could enquire more generally about the class of groups admitting a Cayley complex embeddable in $\R^3$ with no further restrictions, though we do not expect an easy alternative description. An example of such a group 
is the Baumslag--Solitar group $BS(1,2)$. Its standard \Cc\ defined by $\left< a,b \mid ba b^{-1}=a^2 \right>$ embeds in the cartesian product of a binary tree and $\R$, as well-known figures show. It has been proved that $BS(1,2)$ cannot be mapped in a nondegenerate way into the fundamental group of an orientable 3-manifold \cite{ShaThr}, and so it  is not one of the groups of \Tr{inf thm}. 

\section{Appendix: The universal covers of 3-manifolds and orbifolds} \label{sec App}

In this section we provide a proof of \Tr{Agol}. We emphasize that our only contribution to this proof was to ask the experts about it and put the pieces together.
\medskip 


For our proof of \Tr{Agol}, and also in \Sr{sec i to iv}, we make use of Pardon's theorem that topological actions on a 3-manifold can be smoothed:

\begin{theorem}[\cite{Pardon21, PardonMO}] \label{pardon new}
Every \pd\ action of a finitely generated group $\Gamma$ on a 3-manifold $M$ by homeomorphisms is the uniform limit of smooth actions of $\Gamma$ on $M$.
\end{theorem}

We will also need the following consequence of the Orbifold theorem of Boileau, Leeb \& Porti:

\begin{theorem}[{\cite[COROLLARY 1.3]{BoLePoGeo}}] \label{orbi thm}
Every compact connected 3-orbifold which does not contain any bad 2-suborbifolds is the quotient of a compact  3-manifold by a finite group action. \mymargin{\small their 3-orbifold and manifold are oriented, but this can be dropped because a non-orientable orbifold is doubly-covered by an orientable one; and we can make $O$ orientable by using the orientation preserving subgroup of $\Gamma$.}
\end{theorem}

\begin{proof}[Proof of \Tr{Agol}]
Suppose first that $M$ admits an action $\Gamma \act M$ as above, which is in addition free. Then $Q:= M / \Gamma$ is
a closed 3-manifold, and $M$ is its universal cover because it is \sico. Moreover, $\Gamma \isom \pi_1(Q)$.

We may assume \obda\ that $Q$ is orientable by replacing $\Gamma$ by its subgroup of orientation-preserving elements.
Thus $Q$ is a closed, orientable, connected, 3-manifold. 
The fact that such a $Q$ has special universal cover is apparently well-known to experts. We reproduce a proof by Ian Agol \cite{AgolMO}.

If $\pi_1(Q)$ is finite, then its universal cover $M$ is homeomorphic to $\BS^3$ by the validity of the  Poincar\'e conjecture.

If $\pi_1(Q)$ is infinite and  $\pi_2(Q)$ is trivial, then we claim that the universal cover $M$ is homeomorphic to $\R^3$. Indeed, in this case, $Q$ has a geometric decomposition by the \GeThm\ \cite{ThuThr,PerEnt,PerRic}. If the decomposition is trivial, then $Q$ is modelled on one of the six Thurston geometries homeomorphic to $\R^3$, and hence the universal cover is $\R^3$. Otherwise, we apply the Virtually Haken conjecture, proved by Agol \cite{AgoVir},
which asserts that every compact, irreducible 3-manifold $Q$ with infinite $\pi_1(Q)$ is finitely covered by a Haken manifold $Q'$. The reader does not need to know what a  Haken manifold is, all we need is a result of Waldhausen  proving that any Haken manifold $Q'$ has universal cover homeomorphic to $\R^3$ \cite[Theorem 8.1]{WalIrr}. Since $Q'$ covers $Q$, we deduce that $M$ is homeomorphic to $\R^3$ as claimed.

If $\pi_1(Q)$ is infinite and  $\pi_2(Q)$ non-trivial, then it may be that $Q$ is modelled on the $\BS^2 \times \R$ geometry and $Q$ is homeomorphic to $RP3 \# RP3$ or $\BS^2 \times \BS^1$. In this case the number of ends of $\pi_1(Q)$, and $M$, is 2, and $M$ is homeomorphic to $\BS^2 \times \R$. 
Otherwise, $Q$ is a non-trivial connect sum by Papakyriakopoulos' sphere theorem \cite{PapDeh}, and we claim that the universal cover $M$ is a Cantor 3-sphere. Indeed, the connect summands have universal cover either $\BS^3$, or $\BS^2 \times \R$, or $\R^3$ by the above discussion. When forming connect sums, we remove open balls from each summand manifold, and glue their sphere boundaries together. The universal cover is obtained by gluing the universal covers of each summand punctured along balls, either finitely many in $\BS^3$, or infinitely many in $\BS^2 \times \R$ or $\R^3$. It is easy to see that such manifolds are built out of thrice punctured spheres, and hence the universal cover can be decomposed into thrice punctured spheres. It is not hard to see that such a manifold is homeomorphic to the Cantor 3-sphere. 

\medskip
It remains to consider the case where the action $\Gamma \act M$ is non-free, which we will be able to reduce to the free case. By \Tr{pardon new}, we may assume that $\Gamma \act M$ is smooth. Thus the quotient $O:= M / \Gamma$ is endowed with the structure of an orbifold. \Tr{orbi thm} implies that $O$ is finitely covered by a manifold $M'$. The universal cover of $M'$ covers $O$, and so it coincides with $M$ by the uniqueness of a simply connected cover. Recall that  $\pi_1(M')$ acts on its universal cover $M$ freely, \pd ly, and co-compactly. We have reduced to the free case, and so we can deduce that $M$ is special in all cases.
\end{proof}

\acknowledgement{
We thank Ryan Budney, Louis Funar, John Pardon, and Henry Wilton for discussions around \Tr{Agol}, and Ian Agol for providing most of the proof. We thank Johannes Carmesin, Bruno Zimmermann and Mathoverflow user Sam Nead for other helpful discussions.}

\bibliographystyle{plain}
\bibliography{collective}

\end{document}